\documentclass[11pt]{amsart}

\usepackage[T1]{fontenc}
\usepackage[utf8]{inputenc}
\usepackage{lmodern}
\usepackage[margin=1in]{geometry}
\usepackage{amsmath,amssymb,amsthm,mathtools}
\usepackage{mathrsfs}
\usepackage{graphicx}
\usepackage{enumitem}
\usepackage{microtype}
\usepackage{placeins}
\usepackage[colorlinks=true,linkcolor=blue,citecolor=blue,urlcolor=blue]{hyperref}
\usepackage{tikz}
\usetikzlibrary{arrows.meta,positioning,calc,decorations.pathreplacing,fit,shapes.geometric}

\numberwithin{equation}{section}

\newtheorem{theorem}{Theorem}[section]
\newtheorem{proposition}[theorem]{Proposition}
\newtheorem{lemma}[theorem]{Lemma}
\newtheorem{corollary}[theorem]{Corollary}
\newtheorem{definition}[theorem]{Definition}
\newtheorem{setup}[theorem]{Setup}
\newtheorem{remark}[theorem]{Remark}

\newcommand{\Fq}{\mathbb F_q}
\newcommand{\NN}{\mathbb N}
\newcommand{\CC}{\mathbb C}
\newcommand{\Spec}{\operatorname{Spec}}
\newcommand{\diag}{\operatorname{diag}}

\title[Eisenstein scattering and Plancherel decomposition]{Eisenstein Scattering and Plancherel Decomposition\\on Cuspidal Bruhat--Tits Quotients}
\author{Soonki Hong}
\address{Department of Mathematical Education, Catholic Kwandong University, Gangneung 25601, Republic of Korea}
\email{soonki.hong@cku.ac.kr}
\author{Sanghoon Kwon}
\thanks{This work was supported by the Basic Science Research Program through the National Research Foundation of Korea (NRF), funded by the Ministry of Education (grant no. RS-2026-25576727).}
\address{Department of Mathematical Education, Catholic Kwandong University, Gangneung 25601, Republic of Korea}
\email{skwon@cku.ac.kr}
\date{September 22, 2026}
\hypersetup{
  pdftitle={Eisenstein Scattering and Plancherel Decomposition on Cuspidal Bruhat--Tits Quotients},
  pdfauthor={Soonki Hong and Sanghoon Kwon},
  pdfsubject={Eisenstein series, scattering, and spectral decomposition on arithmetic tree quotients},
  pdfkeywords={Bruhat--Tits tree, Eisenstein series, Plancherel formula, Jacobi operator, scattering matrix, Maass--Selberg relation}
}

\begin{document}

\begin{abstract}
For arithmetic quotients of Bruhat--Tits trees with finitely many cusps, we establish an explicit unitary correspondence between the spherical Eisenstein transform, with the Eisenstein series normalized by their constant terms, and the scattering transform of an associated Jacobi operator with finite core. Tracking the Haar measure, stabilizer weights, height coordinates, and cusp widths yields the Plancherel measure and shows that the absolutely continuous spectrum has multiplicity equal to the number of cusps. From a discrete Green identity we derive a matrix-valued Maass--Selberg formula for the Hermitian matrix
$iS(\theta)^*\partial_\theta S(\theta)$,
where $S(\theta)$ is the scattering matrix. Its trace is determined by $\det S(\theta)$, while the full matrix retains additional cusp-to-cusp information.

After the corresponding change of normalization, the finite Schur complement obtained by eliminating the cusp rays agrees with the resonance matrix of Arends-Peterson-Weich. Using their resonance computations as input, we distinguish eigenvalues supported entirely in the finite core from poles of the scattering matrix. The Nagao and $\Gamma_0(T)$ quotients, together with a four-cusp quotient arising from an elliptic curve over $\mathbb F_3$, make the normalizations and matrix-valued conclusions explicit.
\end{abstract}

\maketitle

\setcounter{tocdepth}{1}
\tableofcontents

\section{Introduction}\label{sec:introduction}

Let \(q\) be a prime power, let
\(K=\mathbb F_q(\!(t^{-1})\!)\), and let \(\mathcal T\)
be the \((q+1)\)-regular Bruhat--Tits tree of \(\mathrm{PGL}_2(K)\).  If
\(\Gamma<\mathrm{PGL}_2(K)\) is a non-uniform arithmetic lattice, reduction
theory describes \(Y=\Gamma\backslash\mathcal T\) as a finite graph of
groups with finitely many cuspidal rays
\cite{Nagao,Serre,Harder,BassLubotzky,Bravo}; for harmonic analysis on trees
and groups acting on them, see also \cite{Cartier,FigaTalamancaNebbia}.  The spherical automorphic
Hilbert space is the weighted space \(L^2(Y,\lambda)\), where
\(\lambda(v)=|\Gamma_v|^{-1}\), and the quotient adjacency operator is
self-adjoint for this stabilizer-volume measure.  The precise quotient and
cusp conventions are fixed in Setup~\ref{setup:quotient-adjacency} below.

Classical automorphic theory describes the continuous part by Eisenstein
series, their constant terms, and intertwining operators
\cite{JacquetLanglands,LanglandsEisenstein,LiEisenstein}.  Scattering theory on
graphs gives a complementary finite-dimensional description after the cusps
are cut beyond a finite core
\cite{RomanovRudinBT,RomanovRudinPadic,Novikov,VarbanovBrun,ColinTruc,Golinskii}.
The purpose of this paper is to identify these descriptions with all scalar,
stabilizer, and cusp-width factors visible.  The principal output is an
explicit Eisenstein--Plancherel decomposition; the finite resonance matrix is
used as one means of computing the coefficients, not as the object whose zero
set is to be recomputed.

\subsection{Main results}\label{subsec:main-results}

Write \(Uf=\lambda^{1/2}f\) for the unitary passage from stabilizer measure to
standard \(\ell^2\), and let \(r\) be the number of cusps.  Beyond a finite
set of vertices, the conjugated operator \(J=UAU^{-1}\) is the orthogonal sum
of \(r\) free Jacobi half-lines with
off-diagonal coefficient \(a=\sqrt q\).  If \(\Psi_\theta:\CC^r\to
\mathbb C^{V(Y)}\) denotes the generalized-eigenfunction map whose
incoming Jacobi coefficients are the standard basis vectors, and if
\(P_{\mathrm{ac}}\) denotes the absolutely continuous spectral projection of
\(J\), then
\begin{equation}\label{eq:intro-standard-resolution}
 P_{\mathrm{ac}}
 =\frac1{2\pi}\int_0^\pi\Psi_\theta\Psi_\theta^*\,d\theta,
 \qquad x=2\sqrt q\cos\theta.
\end{equation}
There is no singular continuous spectrum, and the absolutely continuous
multiplicity is the number \(r\) of cusps.  Together with the finite-dimensional
point spectrum, \eqref{eq:intro-standard-resolution} gives a complete spectral
resolution of the spherical adjacency operator.

Let \(\mathscr H_{\mathrm{ac}}=\operatorname{ran}P_{\mathrm{ac}}\).  More
precisely, \(f\mapsto\Psi_\theta^*f\), initially defined for finitely
supported \(f\), extends to a unitary map
\begin{equation}\label{eq:intro-unitary-transform}
 \mathcal F_{\mathrm{ac}}:\mathscr H_{\mathrm{ac}}
 \longrightarrow
 L^2\!\left((0,\pi),\mathbb C^r;\frac{d\theta}{2\pi}\right)
\end{equation}
that conjugates \(J\) to multiplication by \(2\sqrt q\cos\theta\).
The coefficient-level content is strengthened by an exact matrix-valued
Maass--Selberg relation.  If \(P_N\) retains the finite core and the first
\(N\ge1\) vertices of every free cusp, and if
\(S_\theta=S(e^{i\theta})\), then
\begin{equation}\label{eq:intro-maass-selberg}
 \Psi_\theta^*P_N\Psi_\theta
 =(2N+1)I_r+iS_\theta^*\partial_\theta S_\theta
 +\frac{i}{2\sin\theta}
 \left(e^{-i(2N+1)\theta}S_\theta
       -e^{i(2N+1)\theta}S_\theta^*\right).
\end{equation}
The trace of the middle term is determined by \(\det S_\theta\).  When there
is more than one cusp, its traceless part records channel mixing that is not,
in general, determined by the scattering determinant.  Equivalently, the
full middle term is the Ces\`aro-renormalized finite part of the truncated
Gram matrices, whereas the determinant supplies only its trace.

After conjugating back by \(U^{-1}\), write \(P_{\mathrm{ac}}\) also for the
absolutely continuous projection of \(A\).  The same resolution then takes a
particularly simple automorphic form.  Normalize the spherical Eisenstein
family \(\mathcal E_s\) by requiring its incoming
constant term at the \(c\)-th cusp to be \(q^{ns}e_c\), where \(e_c\) is the
\(c\)-th standard vector of \(\mathbb C^r\), and put
\[
 s=\frac12+\frac{i\theta}{\log q},
 \qquad W=\diag(w_1,\ldots,w_r),
\]
where \(w_c\) is the width determined by the chosen height origin at that cusp.
Then
\begin{equation}\label{eq:intro-automorphic-resolution}
 P_{\mathrm{ac}}
 =\int_0^\pi \mathcal E_s\,
   \frac{1}{2\pi(q+1)}W^{-1}\,
   \mathcal E_s^*\,d\theta.
\end{equation}
The exact conversion to Efrat's holomorphic local factors gives its
nonconstant density from the constant matrix in
\eqref{eq:intro-automorphic-resolution}.  Product-formula, constant-term, and
double-coset arguments identify this graph family with the fixed-level adelic
Eisenstein family, including the incoming coefficient and every finite-level
measure factor.  In the notation of Setup~\ref{setup:global-automorphic}, the
resulting transform is the unitary equivalence
\begin{equation}\label{eq:intro-adelic-unitary-identification}
 \bigoplus_{j=1}^h L^2_{\mathrm{Eis}}(Y_j)
 \ \xrightarrow{\ \sim\ }\
 L^2\!\left((0,\pi),
   \mathbb C^{\sum_j r_j};\frac{d\theta}{2\pi}\right).
\end{equation}
Here \(h\) is the number of fixed-level components,
\(Y_j=\Gamma_j\backslash\mathcal T\), and \(r_j\) is the number of cusps of
\(Y_j\),
and the unnormalized spherical Hecke operator at \(\infty\) becomes
multiplication by \(2\sqrt q\cos\theta\); see
Corollary~\ref{cor:adelic-unitary-identification}.

The coefficients in these expansions are computed from a finite matrix.  If
the core has \(d\) vertices, \(H_C\) is its \(d\times d\) Hermitian matrix,
and \(V:\CC^r\to\CC^d\) records the labelled cusp
attachments, set
\begin{equation}\label{eq:intro-finite-matrix}
 x(\zeta)=a(\zeta+\zeta^{-1}),
 \qquad
 F(\zeta)=H_C-x(\zeta)I_d+\frac{\zeta^{-1}}aVV^*.
\end{equation}
For \(\zeta=q^{s-1/2}\), the Jacobi-normalized scattering matrix is
\begin{equation}\label{eq:intro-S}
 S(\zeta)=-I_r+\frac{\zeta^{-1}-\zeta}{a}
 V^*F(\zeta)^{-1}V.
\end{equation}
Theorem~\ref{thm:normalization-dictionary} converts this matrix into the
classical constant-term matrix.  Theorem~\ref{thm:resolvent-jump} derives
\eqref{eq:intro-standard-resolution} directly from the jump of the resolvent,
and Theorem~\ref{thm:automorphic-plancherel} gives
\eqref{eq:intro-automorphic-resolution}.

For cusp-only graphs of groups, \(F\) is exactly \(2\sqrt q\) times the finite
resonance matrix of Arends, Peterson, and Weich after stabilizer conjugation
\cite{ArendsPetersonWeich}.  The labelled attachment map recovers the full
constant-term matrix, while core-supported eigenfactors cancel from its
determinant.  The \(\Gamma_0(T)\) and four-cusp elliptic examples make these
coefficient-level and Plancherel consequences explicit.

Figure~\ref{fig:proof-architecture} summarizes the passage from the
arithmetic quotient to the two principal spectral consequences.

\begin{figure}[htbp]
\centering
\begin{tikzpicture}[
 flowbox/.style={draw,rounded corners=2pt,fill=black!3,
   align=center,inner sep=4.5pt,font=\small},
 flowarrow/.style={-{Latex[length=2.2mm]},semithick},
 arrowlabel/.style={font=\scriptsize,fill=white,inner sep=1.2pt}
]
\node[flowbox,text width=.58\textwidth] (quotient)
 {Arithmetic quotient\\[-1pt]
  \(L^2(Y,\lambda)\) with adjacency operator \(A\)};
\node[flowbox,text width=.70\textwidth,below=5mm of quotient] (jacobi)
 {
  Stabilizer conjugation \(U{\,=\,}\lambda^{1/2}\)\\[-1pt]
  finite core \(H_C\), labelled attachments \(V\), and \(r\) free Jacobi tails};
\node[flowbox,text width=.56\textwidth,below=5mm of jacobi] (finite)
 {Finite Schur complement\\[-1pt]
  \(F(\zeta)\), scattering matrix \(S(\zeta)\), and generalized eigenfunctions
  \(\Psi(\zeta)\)};
\node[flowbox,text width=.30\textwidth,below=8mm of finite,
      xshift=-.185\textwidth] (jump)
 {Resolvent jump\\[-1pt]
  unitary spectral transform and Plancherel resolution};
\node[flowbox,text width=.30\textwidth,below=8mm of finite,
      xshift=.185\textwidth] (green)
 {Discrete Green identity\\[-1pt]
  matrix-valued Maass--Selberg relation};
\coordinate (lowercenter) at ($(jump.south)!0.5!(green.south)$);
\node[flowbox,text width=.70\textwidth,below=8mm of lowercenter] (automorphic)
 {Automorphic normalization\\[-1pt]
  cusp-width matrix \(W\), constant-term-normalized family \(\mathcal E_s\),
  and exact spectral decomposition};

\draw[flowarrow] (quotient) -- node[arrowlabel,right=1mm] {\(U\)} (jacobi);
\draw[flowarrow] (jacobi) -- node[arrowlabel,right=1mm] {finite reduction} (finite);
\draw[flowarrow] (finite.south) -- ++(0,-3mm) -| (jump.north);
\draw[flowarrow] (finite.south) -- ++(0,-3mm) -| (green.north);
\draw[flowarrow] (jump.south) -- ($(automorphic.north)+(-.18\textwidth,0)$);
\draw[flowarrow] (green.south) -- ($(automorphic.north)+(.18\textwidth,0)$);
\end{tikzpicture}
\caption{Logical architecture of the paper.  The finite Schur complement
supplies both the scattering states and their coefficient matrix.  Its
resolvent jump yields the Plancherel resolution, while the discrete Green
identity yields the matrix-valued Maass--Selberg relation; the final row
restores the automorphic normalization.}
\label{fig:proof-architecture}
\end{figure}
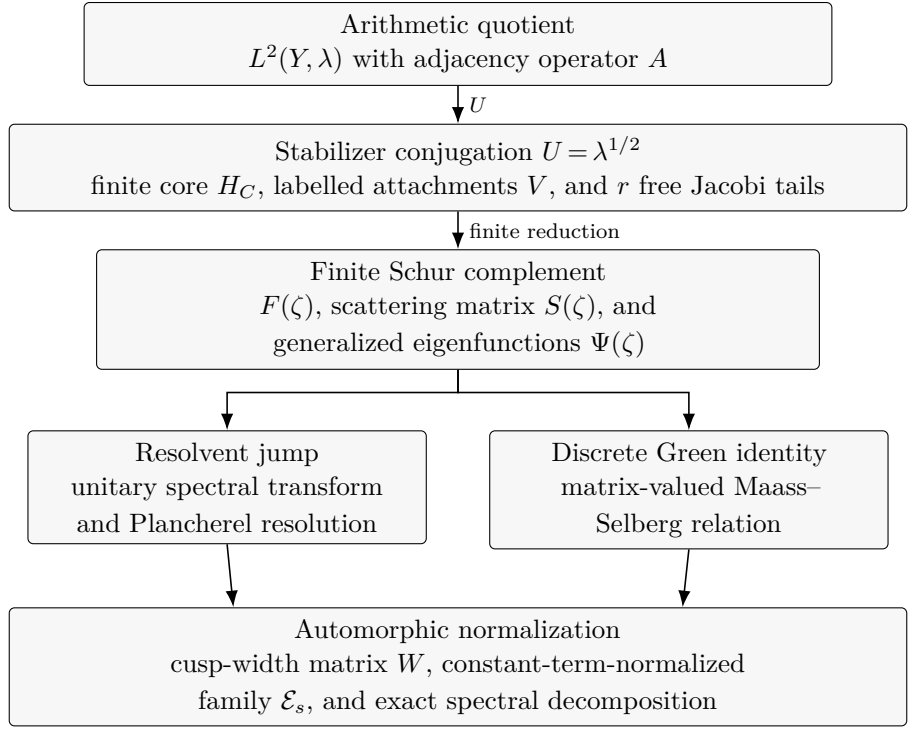

\subsection{Earlier work and precise scope}\label{subsec:intro-related-work}

The archimedean automorphic scattering paradigm linking Eisenstein series,
scattering matrices, and spectral resolution is developed by Lax and Phillips
\cite{LaxPhillips}.  In the function-field graph setting, Efrat studied
spectral deformations of automorphic functions on graphs of groups
\cite{EfratSpectralDeformations} and later computed the discrete and continuous
spectrum of the Nagao ray, including generalized eigenfunctions, Eisenstein
parameters, and a Plancherel formula \cite{EfratAutomorphicSpectra}.  Nagoshi
also described the discrete--continuous \(L^2\)-decomposition of arithmetic
infinite graphs and the Eisenstein series furnishing their continuous spectrum
\cite[Sec.~2]{NagoshiSpectra}.

Paulin developed an intrinsic geometric framework for non-uniform tree
lattices through quotient graphs of groups and their geodesic flow
\cite{PaulinGeometricallyFinite}.  From the closed-geodesic viewpoint,
Deitmar--Kang and Hong--Kwon studied zeta functions, determinants, and
geodesic counting on noncompact tree quotients
\cite{DeitmarKang,HongKwonZeta}.  Those geometric zeta invariants are distinct
from the Eisenstein coefficient matrices considered here.

Finite-tail scattering matrices, their unitarity, and a resolution of the
identity already appear in Varbanov and Brun
\cite[especially Eq.~(32)]{VarbanovBrun}.  Colin de Verdi\`ere and Truc give an
explicit scattering-theoretic spectral decomposition for graphs that are
regular trees outside a finite set \cite{ColinTruc}.  Their branching-tree
ends differ from the stabilizer-weighted cusp half-lines considered here, but
their work is an important spectral antecedent.  Chekhov relates scattering
determinants on \(p\)-adic graphs to finite determinants and Ihara--Selberg
\(L\)-functions and also analyzes discrete spectral factors
\cite{Chekhov,ChekhovSurvey}.  Accordingly, neither a finite-tail scattering
formula nor a reciprocal determinant identity is claimed as new here.

The Nagao calculation in Subsection~\ref{subsec:nagao-calibration} calibrates
the present normalizations.  In Efrat's measure, the \((0,0)\)-entry of the
identity kernel gives the coefficient \(2/\pi\), replacing the printed
factor \(2\pi\) in \cite[Thm.~5.3, Eq.~(7)]{EfratAutomorphicSpectra}; the
short check is recorded there.

Matrices of the form \(iS^*\partial_\theta S\) are classical in scattering
theory \cite{SmithLifetime,ReedSimonIII,Yafaev}, as is the existence of the
automorphic spectral decomposition.  The contribution here is their
normalization-exact identification in the stabilizer-weighted arithmetic
setting: the labelled coefficient matrix, the exact cusp-width Plancherel
measure, the unitary fixed-level Eisenstein transform, and the automorphic
matrix Green--Maass--Selberg identity.  Arends, Peterson, and Weich provide
the meromorphic resolvent, resonant-state theory, finite resonance matrix, and
explicit resonance sets \cite[Secs.~1 and~7.8]{ArendsPetersonWeich}; the
present paper takes these finite resonance data as input and derives the
complementary unitary and coefficient-level consequences.

\subsection{Organization and terminology}\label{subsec:intro-organization}

Section~\ref{sec:finite-core} proves the general
scattering, Green--Maass--Selberg, and unitary Plancherel formulas and returns
them to automorphic notation.
Section~\ref{sec:cusp-invisible} separates the core-supported point spectrum.
Section~\ref{sec:resonance-comparison} compares the finite-core matrix with
the resonance matrix of Arends, Peterson, and Weich.
Section~\ref{sec:examples} treats the Nagao quotient, \(\Gamma_0(T)\), and the
four-cusp elliptic quotient.  The final section records the scope.

The inner product is linear in the first variable.  A \emph{cusp channel}
means only the homogeneous half-line remaining after a finite initial segment
of a cusp has been absorbed into the core.  The symbols \(A\) and
\(J=UAU^{-1}\) denote the adjacency operator in stabilizer measure and its
standard \(\ell^2\) realization.  We reserve \(a=\sqrt q\) for the free Jacobi
coefficient, \(\zeta=q^{s-1/2}\) for the multiplicative spectral parameter,
\(W\) for the cusp-width matrix, and \(P_{\mathrm{ac}}\) for the absolutely
continuous spectral projection.  The symbols
\(E_s^{\mathrm{hol}}\) and \(\mathcal E_s\) denote the multi-cusp
holomorphic and constant-term-normalized Eisenstein families, respectively.
For vectors or generalized eigenfunctions \(v,w\), the symbol \(vw^*\)
denotes the rank-one sesquilinear form
\(f\mapsto\langle f,w\rangle v\).  Expressions involving generalized
eigenfunctions are first interpreted on finitely supported vectors and are
then extended in the stated \(L^2\) sense.
For a positive integer \(m\), we write \(M_m(\mathbb C)\) for the
\(m\times m\) complex matrices,
\(v^t\) for transpose, \(\operatorname{ran}\) for range, and a dot over a
matrix depending on \(\theta\) for \(\partial_\theta\).

\section{Finite-core reduction and Plancherel decomposition}\label{sec:finite-core}

We now prove the general spectral decomposition.  Passing first to standard
\(\ell^2\) turns the homogeneous parts of the cusps into free Jacobi
half-lines.  The finite core then determines both the Eisenstein coefficients
and the spectral measure, without hidden diagonal factors.

\begin{setup}[Quotient measure, adjacency, and cusp widths]
\label{setup:quotient-adjacency}
Put \(G=\mathrm{PGL}_2(K)\),
\(\mathcal O=\mathbb F_q[\![t^{-1}]\!]\), and
\(\mathcal K=\mathrm{PGL}_2(\mathcal O)\).  Identify
\(V(\mathcal T)=G/\mathcal K\), normalize Haar measure by
\(\operatorname{vol}(\mathcal K)=1\), and write
\(Y=\Gamma\backslash\mathcal T\).  For a quotient vertex \(v\), choose a
lift \(\widetilde v\) and set
\[
 \Gamma_v=\operatorname{Stab}_\Gamma(\widetilde v),
 \qquad \lambda(v)=|\Gamma_v|^{-1}.
\]
Then
\[
 L^2(\Gamma\backslash G/\mathcal K)\simeq L^2(Y,\lambda),
 \qquad
 \|f\|^2=\sum_{v\in V(Y)}|f(v)|^2\lambda(v).
\]
Let \(\mathscr E_v\) represent the \(\Gamma_v\)-orbits of oriented tree
edges issuing from \(\widetilde v\).  For \(e\in\mathscr E_v\), let
\(\Gamma_e=\operatorname{Stab}_{\Gamma_v}(e)\) and let \(t(e)\) be its
terminal quotient vertex.  The quotient adjacency operator is
\begin{equation}\label{eq:quotient-adjacency}
 (Af)(v)=\sum_{e\in\mathscr E_v}[\Gamma_v:\Gamma_e]f(t(e)).
\end{equation}
This also gives quotient loops their directed multiplicities.  If \(\bar e\)
is the reversed edge orbit, orbit--stabilizer gives
\[
 \lambda(v)[\Gamma_v:\Gamma_e]
 =|\Gamma_e|^{-1}
 =\lambda(t(e))[\Gamma_{t(e)}:\Gamma_{\bar e}].
\]
Since the directed indices from each vertex sum to \(q+1\), the weighted
Schur test makes \(A\) a bounded self-adjoint operator with
\(\|A\|\le q+1\).

Cut every cusp in its homogeneous part and number the retained vertices by
\((c,n)\), \(n\ge1\), increasing away from the core.  Its width relative to
this height origin is
\begin{equation}\label{eq:cusp-width-definition}
 w_c=\frac{q^n\lambda(c,n)}{q+1},
 \qquad\text{equivalently}\qquad
 \lambda(c,n)=w_c(q+1)q^{-n}.
\end{equation}
Homogeneity makes \(w_c\) independent of \(n\); its covariance under a
deeper cut is recorded in Remark~\ref{rem:height-shift}.
\end{setup}

\subsection{Standard Jacobi realization}\label{subsec:standard-channel-model}

Let
\[
        \mathscr H=\CC^d\oplus\bigoplus_{c=1}^r\ell^2(\NN),
        \qquad \NN=\{1,2,\ldots\},
\]
where \(d,r\ge1\), and write \(e_c\) for the \(c\)-th standard vector of \(\CC^r\).
Let \(J_0^{(r)}\) be the orthogonal sum of \(r\) copies of
\[
        (J_0g)(1)=ag(2),\qquad
        (J_0g)(n)=a\{g(n-1)+g(n+1)\}\quad(n\ge2),
        \qquad a=\sqrt q.
\]
We use the standard half-line Jacobi normalization; see, for example,
\cite{Teschl}.
Write \(\Gamma_1:\CC^r\to\bigoplus_{c=1}^r\ell^2(\NN)\) for insertion at the first vertices,
\[
        (\Gamma_1\alpha)_c(n)=\delta_{n1}\alpha_c.
\]
Let \(H_C=H_C^*\in M_d(\CC)\), and let \(V:\CC^r\to\CC^d\).  The coupled operator is
\begin{equation}\label{eq:block-J}
        J=
        \begin{pmatrix}
        H_C&V\Gamma_1^*\\
        \Gamma_1V^*&J_0^{(r)}
        \end{pmatrix}.
\end{equation}
Thus, if \(u\in\CC^d\) and \(g=(g_c)_{c=1}^r\), then
\[
        J(u,g)=\bigl(H_Cu+Vg(1),\,J_0^{(r)}g+\Gamma_1V^*u\bigr).
\]
Here and below, \(g(1)=(g_1(1),\ldots,g_r(1))^t\).

Figure~\ref{fig:finite-core-channels} makes the role of the attachment map
and the cusp labels explicit.

\begin{figure}[htbp]
\centering
\begin{tikzpicture}[
 core/.style={draw,rounded corners=3pt,fill=black!5,
   minimum width=25mm,minimum height=24mm,align=center,font=\small},
 tail/.style={circle,draw,fill=white,minimum size=9mm,inner sep=1pt,
   font=\scriptsize},
 edge/.style={semithick},
 edgelabel/.style={font=\scriptsize,fill=white,inner sep=1pt}
]
\node[core] (core) at (-3.1,0) {finite core\\\(H_C\)};

\node[tail] (c11) at (-.35,1.2) {\((1,1)\)};
\node[tail] (c12) at (1.25,1.2) {\((1,2)\)};
\node (dots1) at (2.65,1.2) {\(\cdots\)};

\node[tail] (c21) at (-.35,0) {\((2,1)\)};
\node[tail] (c22) at (1.25,0) {\((2,2)\)};
\node (dots2) at (2.65,0) {\(\cdots\)};

\node[tail] (cr1) at (-.35,-1.2) {\((r,1)\)};
\node[tail] (cr2) at (1.25,-1.2) {\((r,2)\)};
\node (dotsr) at (2.65,-1.2) {\(\cdots\)};

\draw[edge] (core.north east) -- node[edgelabel,above] {\(Ve_1\)} (c11.west);
\draw[edge] (core.east) -- node[edgelabel,above] {\(Ve_2\)} (c21.west);
\draw[edge] (core.south east) -- node[edgelabel,below] {\(Ve_r\)} (cr1.west);

\draw[edge] (c11) -- node[edgelabel,above] {\(a\)} (c12);
\draw[densely dotted,semithick] (c12) -- (dots1);
\draw[edge] (c21) -- node[edgelabel,above] {\(a\)} (c22);
\draw[densely dotted,semithick] (c22) -- (dots2);
\draw[edge] (cr1) -- node[edgelabel,above] {\(a\)} (cr2);
\draw[densely dotted,semithick] (cr2) -- (dotsr);

\node[font=\small,align=center] at (0,-2.0)
 {\(a=\sqrt q\),\qquad
  \(\lambda(c,n)=w_c(q+1)q^{-n}\) before conjugation};
\end{tikzpicture}
\caption{The standard-\(\ell^2\) finite-core model.  The column \(Ve_c\)
couples the labelled channel \(c\) to the core, and every remaining cusp tail
has the free Jacobi coefficient \(a=\sqrt q\).  Intermediate channels between
\(2\) and \(r\) are omitted from the drawing.}
\label{fig:finite-core-channels}
\end{figure}
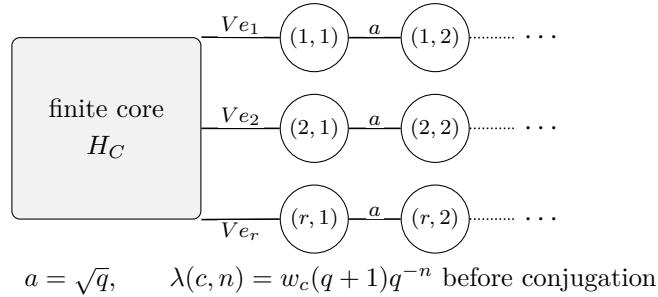
\FloatBarrier

\begin{proposition}[Unitary normal form]\label{prop:unitary-normal-form}
Let a stabilizer-weighted graph consist of a finite part and \(r\) homogeneous cuspidal rays.  Suppose that, after a finite truncation, the \(c\)-th ray has
\[
        (Af)(c,n)=qf(c,n-1)+f(c,n+1),
        \qquad
        \lambda(c,n)=w_c(q+1)q^{-n}.
\]
Then \(U f=\lambda^{1/2}f\) conjugates \(A\) to an operator of the form
\eqref{eq:block-J}.  All non-homogeneous initial cusp vertices may be absorbed
into \(H_C\), and the remaining attachments to the homogeneous cusp rays are
contained in \(V\).
\end{proposition}

\begin{proof}
On a homogeneous ray,
\[
\begin{aligned}
 (UAU^{-1}g)(c,n)
 &=\lambda(c,n)^{1/2}
   \left(
   q\frac{g(c,n-1)}{\lambda(c,n-1)^{1/2}}
   +\frac{g(c,n+1)}{\lambda(c,n+1)^{1/2}}
   \right)\\
 &=\sqrt q\,g(c,n-1)+\sqrt q\,g(c,n+1).
\end{aligned}
\]
The complement of the homogeneous tails is finite-dimensional.  Self-adjointness makes its compression Hermitian and makes the two off-diagonal blocks adjoints of one another, giving \eqref{eq:block-J}.
\end{proof}

Put
\begin{equation}\label{eq:x-zeta}
        x(\zeta)=a(\zeta+\zeta^{-1}),
        \qquad \zeta\in\CC^\times.
\end{equation}
For \(x=x(\zeta)\), the tail recurrence has the two solutions \(\zeta^n\) and \(\zeta^{-n}\).  We call \(\zeta^n\) incoming and \(\zeta^{-n}\) outgoing.  This convention agrees on \(|\zeta|=1\), with \(\zeta=e^{i\theta}\), with waves directed toward and away from the finite core, respectively.
The two modes coalesce at \(\zeta=\pm1\); these two branch points of
\(\zeta\mapsto x(\zeta)\) are called the \emph{threshold parameters}.

\subsection{Scattering coefficients from the finite core}\label{subsec:finite-scattering}

\begin{definition}[Finite outgoing matrix and determinantal parameters]\label{def:finite-outgoing-matrix}
For the standard model \eqref{eq:block-J}, define
\begin{equation}\label{eq:F-zeta}
        F(\zeta)
        =H_C-x(\zeta)I_d+\frac{\zeta^{-1}}aVV^*.
\end{equation}
The resolvent family \((J-x(\zeta))^{-1}\) is initially defined for
\(|\zeta|\) sufficiently large.  It extends meromorphically throughout the
branch \(|\zeta|>1\) selected by the square-summable mode \(\zeta^{-n}\), and
then in \(\zeta\in\mathbb C^\times\).
The matrix
\(F(\zeta)\) is the Schur complement obtained by imposing the outgoing
solution \(\zeta^{-n}\) on each cusp.  Its determinant also records the
finite-core resonance parameters, with the threshold points treated
separately.
For \(\zeta_0\in\mathbb C^\times\setminus\{\pm1\}\), we call
\(\zeta_0\) a \emph{determinantal resonance parameter} when
\(\det F(\zeta_0)=0\), and call
\(\operatorname{ord}_{\zeta_0}\det F\) its determinantal order.  In the
cusp-only graph-of-groups setting, Theorem~\ref{thm:finite-resonance-comparison}
identifies these parameters with the resonances of
Arends, Peterson, and Weich.  On the open unit-circle arc
\(\zeta=e^{i\theta}\), set
\begin{equation}\label{eq:singular-unitary-set}
 \Theta_F=\{\theta\in(0,\pi):\det F(e^{i\theta})=0\}.
\end{equation}
This is a finite set.  Values of meromorphic families at removable points of
\(\Theta_F\) are always understood by continuation.
\end{definition}

\begin{lemma}[Outgoing Schur complement]\label{lem:exact-schur-complement}
For \(|\zeta|>1\),
\[
        \Gamma_1^*(J_0^{(r)}-x(\zeta))^{-1}\Gamma_1
        =-\frac{\zeta^{-1}}aI_r,
\]
and the core--core block of \((J-x(\zeta))^{-1}\) is \(F(\zeta)^{-1}\).  These identities extend meromorphically wherever the indicated inverses exist.
\end{lemma}

\begin{proof}
The square-summable root of the free tail recurrence is \(\zeta^{-1}\).
The first-vertex Weyl function
\[
 m(z)=\langle\delta_1,(J_0-z)^{-1}\delta_1\rangle
\]
where \(\delta_1\) is the first standard basis vector of \(\ell^2(\mathbb N)\).
It satisfies, at \(z=x(\zeta)\),
\[
        m=\frac1{-x(\zeta)-a^2m},
\]
and the solution with \(m\sim-1/x\) is \(m=-\zeta^{-1}/a\).  Taking the Schur complement of the tail block in \(J-x(\zeta)\) now gives
\[
\begin{aligned}
 H_C-x(\zeta)I_d
 -V\Gamma_1^*(J_0^{(r)}-x(\zeta))^{-1}\Gamma_1V^*
 &=H_C-x(\zeta)I_d+\frac{\zeta^{-1}}aVV^*\\
 &=F(\zeta).
\end{aligned}
\]
The block inverse formula proves the assertion.
\end{proof}

Equivalently, a purely outgoing solution has \(g(n)=\zeta^{-n}\beta\).  The first tail equation gives \(V^*u=a\beta\), and the core equation becomes \(F(\zeta)u=0\).  This explains directly why the coefficient in \eqref{eq:F-zeta} is \(\zeta^{-1}/a\), rather than an unspecified Weyl function.

\begin{theorem}[Scattering matrix from the finite core]\label{thm:exact-scattering}
Assume that \(F(\zeta)\) is invertible and \(\zeta\ne\pm1\).  For every incoming amplitude \(\alpha\in\CC^r\), there is a unique generalized eigenfunction
\[
        \Psi(\zeta)\alpha=(u_\alpha,g_\alpha)
\]
of \(J\) at \(x(\zeta)\), with
\begin{equation}\label{eq:standard-scattering-asymptotic}
        g_\alpha(n)=\zeta^n\alpha+\zeta^{-n}S(\zeta)\alpha.
\end{equation}
It is given by
\begin{align}
        u_\alpha
        &=(\zeta^{-1}-\zeta)F(\zeta)^{-1}V\alpha,
        \label{eq:core-poisson}\\
        S(\zeta)
        &=-I_r+\frac{\zeta^{-1}-\zeta}{a}
        V^*F(\zeta)^{-1}V.
        \label{eq:S-from-F}
\end{align}
Consequently \(\Psi(\zeta)\) and \(S(\zeta)\) are meromorphic matrix-valued functions of \(\zeta\).
\end{theorem}

\begin{proof}
Write the outgoing amplitude as \(\beta\), so that
\[
        g(n)=\zeta^n\alpha+\zeta^{-n}\beta.
\]
The eigenvalue equation at the first vertex of every tail is
\[
        ag(2)+V^*u=x(\zeta)g(1).
\]
Using \eqref{eq:x-zeta} gives the exact boundary relation
\begin{equation}\label{eq:boundary-relation}
        V^*u=a(\alpha+\beta).
\end{equation}
The core equation is
\[
        (H_C-x(\zeta)I)u+V(\zeta\alpha+\zeta^{-1}\beta)=0.
\]
Eliminating \(\beta=a^{-1}V^*u-\alpha\) by \eqref{eq:boundary-relation} yields
\[
        F(\zeta)u=(\zeta^{-1}-\zeta)V\alpha.
\]
This proves \eqref{eq:core-poisson}; substituting it back into
\(\beta=a^{-1}V^*u-\alpha\) proves \eqref{eq:S-from-F}.  All entries are rational functions of \(\zeta\) and the entries of \(F(\zeta)^{-1}\), proving meromorphicity.
\end{proof}

\begin{corollary}[Functional equation and unitarity]\label{cor:S-functional-unitary}
As meromorphic identities,
\begin{equation}\label{eq:S-functional}
        S(\zeta^{-1})S(\zeta)=I_r.
\end{equation}
For \(\zeta=e^{i\theta}\) with
\(\theta\in(0,\pi)\setminus\Theta_F\),
\begin{equation}\label{eq:S-unitary}
        S(\zeta)^*=S(\zeta)^{-1}=S(\zeta^{-1}).
\end{equation}
\end{corollary}

\begin{proof}
Replacing \(\zeta\) by \(\zeta^{-1}\) exchanges the incoming and outgoing modes.  Uniqueness in Theorem \ref{thm:exact-scattering} therefore gives \eqref{eq:S-functional} wherever both sides are regular, hence meromorphically everywhere.  On the unit circle,
\[
        F(\zeta)^*=F(\zeta^{-1}).
\]
Taking adjoints in \eqref{eq:S-from-F} gives \(S(\zeta)^*=S(\zeta^{-1})\), and \eqref{eq:S-functional} proves unitarity.
\end{proof}

\subsection{A matrix-valued Green--Maass--Selberg identity}
\label{subsec:matrix-maass-selberg}

The full coefficient matrix also controls truncated inner products, not only
the pole set or the total scattering phase.  Let \(P_N\) be the orthogonal
projection of \(\mathscr H\) onto the finite core and the vertices
\(1,\ldots,N\) in every free tail, where \(N\ge1\).  For
\(\theta\in(0,\pi)\setminus\Theta_F\), write
\[
 x_\theta=2a\cos\theta,\qquad
 S_\theta=S(e^{i\theta}),\qquad
 G_\theta(n)=e^{in\theta}I_r+e^{-in\theta}S_\theta.
\]
Thus \(G_\theta(n)\) is the tail-value map of
\(\Psi_\theta=\Psi(e^{i\theta})\) at level \(n\).
The next theorem is the finite-difference, matrix-valued analogue of the
classical Maass--Selberg relation for truncated Eisenstein series
\cite{LanglandsEisenstein}.

\begin{theorem}[Matrix Green and Maass--Selberg identities]
\label{thm:matrix-maass-selberg}
For \(N\ge1\) and
\(\theta,\phi\in(0,\pi)\setminus\Theta_F\) with \(\theta\ne\phi\),
\begin{equation}\label{eq:matrix-green-identity}
\begin{aligned}
 (x_\theta-x_\phi)\Psi_\theta^*P_N\Psi_\phi
 =a\bigl(&G_\theta(N+1)^*G_\phi(N)\\
          &-G_\theta(N)^*G_\phi(N+1)\bigr).
\end{aligned}
\end{equation}
For every \(\theta\in(0,\pi)\setminus\Theta_F\), the equal-parameter
boundary identity gives
\begin{equation}\label{eq:green-flux-unitarity}
 S_\theta^*S_\theta=I_r.
\end{equation}
Moreover, with \(\dot S_\theta=\partial_\theta S(e^{i\theta})\),
\begin{equation}\label{eq:matrix-maass-selberg}
 \Psi_\theta^*P_N\Psi_\theta
 =(2N+1)I_r+iS_\theta^*\dot S_\theta
 +\frac{i}{2\sin\theta}
 \left(e^{-i(2N+1)\theta}S_\theta
       -e^{i(2N+1)\theta}S_\theta^*\right).
\end{equation}
In particular, every term on the right-hand side of
\eqref{eq:matrix-maass-selberg} is Hermitian.
\end{theorem}

\begin{proof}
Fix \(\alpha,\beta\in\mathbb C^r\), put
\(f=\Psi_\phi\alpha\) and \(h=\Psi_\theta\beta\), and write
\(f_c(n)=f(c,n)\), \(h_c(n)=h(c,n)\) for their \(c\)-th tail coordinates.
Sum
\(\langle Jf,h\rangle-\langle f,Jh\rangle\) over the finite core and the
first \(N\) levels of all tails.  All core, attachment, and interior-tail
terms cancel.  The remaining boundary term is
\[
 a\sum_{c=1}^r
 \bigl(f_c(N+1)\overline{h_c(N)}
       -f_c(N)\overline{h_c(N+1)}\bigr).
\]
For each fixed \(c\), the corresponding summand is precisely the
sesquilinear flux through the cut edge joining levels \(N\) and \(N+1\) in
that cusp.  Thus the displayed sum is the total boundary flux of the
truncated graph; no additional boundary form is being suppressed.
Since \(Jf=x_\phi f\), \(Jh=x_\theta h\), and the spectral values are real,
reversing both sides is exactly the scalar pairing of \(\alpha,\beta\) with
\eqref{eq:matrix-green-identity}.

At equal parameters the left-hand side vanishes.  Expanding the boundary
term gives
\[
 a(e^{-i\theta}-e^{i\theta})(I_r-S_\theta^*S_\theta),
\]
and \(0<\theta<\pi\) proves \eqref{eq:green-flux-unitarity}.  Finally, let
\(\phi\to\theta\) in \eqref{eq:matrix-green-identity}.  Since
\(-\partial_\theta x_\theta=2a\sin\theta\), differentiating the boundary
term in \(\phi\) and using \(S_\theta^*S_\theta=I_r\) gives
\[
\begin{aligned}
 2\sin\theta\,\Psi_\theta^*P_N\Psi_\theta
 ={}&2(2N+1)\sin\theta\,I_r
      +2i\sin\theta\,S_\theta^*\dot S_\theta\\
 &+i e^{-i(2N+1)\theta}S_\theta
   -i e^{i(2N+1)\theta}S_\theta^*.
\end{aligned}
\]
Division by \(2\sin\theta\) proves
\eqref{eq:matrix-maass-selberg}.  Differentiating unitarity shows that
\(S_\theta^*\dot S_\theta\) is skew-Hermitian, which proves the final
assertion.
\end{proof}

\begin{corollary}[Coefficient information beyond the determinant]
\label{cor:maass-selberg-traceless}
For \(\theta\in(0,\pi)\setminus\Theta_F\), define the Hermitian
Maass--Selberg logarithmic-derivative matrix
\begin{equation}\label{eq:channel-mass-matrix}
 \mathcal M(\theta)=iS_\theta^*\dot S_\theta
\end{equation}
and its traceless part
\begin{equation}\label{eq:traceless-channel-mass}
 \mathcal M_0(\theta)
 =\mathcal M(\theta)-\frac1r\operatorname{tr}\mathcal M(\theta)I_r.
\end{equation}
Then \(\mathcal M(\theta)\) is recovered from any truncated Gram matrix in
\eqref{eq:matrix-maass-selberg} after subtracting the displayed universal and
boundary terms.  It also has the averaged form
\begin{equation}\label{eq:renormalized-gram-limit}
 \mathcal M(\theta)
 =\lim_{L\to\infty}\frac1L\sum_{N=1}^{L}
 \left\{\Psi_\theta^*P_N\Psi_\theta-(2N+1)I_r\right\},
\end{equation}
where the limit is taken in \(M_r(\mathbb C)\).  Locally on the unitary axis,
\begin{equation}\label{eq:trace-mass-determinant}
 \operatorname{tr}\mathcal M(\theta)
 =i\,\partial_\theta\log\det S_\theta.
\end{equation}
Thus \(\det S_\theta\) determines only the trace of
\(\mathcal M(\theta)\).  For \(r\ge2\), it does not in general determine
\(\mathcal M_0(\theta)\), which retains the channel-resolved finite term in
the truncated Gram matrix; for \(r=1\), the traceless part is zero.
\end{corollary}

\begin{proof}
Jacobi's formula and unitarity give
\[
 \partial_\theta\log\det S_\theta
 =\operatorname{tr}(S_\theta^{-1}\dot S_\theta)
 =\operatorname{tr}(S_\theta^*\dot S_\theta).
\]
The geometric averages of the two boundary exponentials in
\eqref{eq:matrix-maass-selberg} tend to zero because
\(e^{2i\theta}\ne1\), proving \eqref{eq:renormalized-gram-limit}.  The other
statements follow directly from \eqref{eq:matrix-maass-selberg}.  The genuine
two-cusp calculation in Corollary~\ref{cor:gamma0-channel-mass} exhibits a
nonzero traceless part and therefore realizes the asserted loss of
information arithmetically.
\end{proof}

The next identity determines the total scattering phase from a single finite
determinant.  Reciprocal determinant principles in \(p\)-adic graph scattering
predate the present work \cite{Chekhov}; here the identity fixes the precise
matrix needed for comparison with Eisenstein constant terms and the finite
resonance formalism of Arends, Peterson, and Weich.

\begin{theorem}[Reciprocal determinant identity]\label{thm:determinant-identity}
Whenever \(\det F(\zeta)\det F(\zeta^{-1})\ne0\),
\begin{equation}\label{eq:determinant-identity}
        \det S(\zeta)
        =(-1)^r\frac{\det F(\zeta^{-1})}{\det F(\zeta)}.
\end{equation}
Both sides agree as meromorphic functions.
\end{theorem}

\begin{proof}
Set \(c(\zeta)=(\zeta^{-1}-\zeta)/a\).  Sylvester's determinant identity and \eqref{eq:S-from-F} give
\[
\begin{aligned}
 \det S(\zeta)
 &=(-1)^r\det\!\left(I_r-c(\zeta)V^*F(\zeta)^{-1}V\right)\\
 &=(-1)^r\det\!\left(I_d-c(\zeta)F(\zeta)^{-1}VV^*\right)\\
 &=(-1)^r\frac{\det(F(\zeta)-c(\zeta)VV^*)}{\det F(\zeta)}.
\end{aligned}
\]
But the definition of \(F\) gives
\[
        F(\zeta)-c(\zeta)VV^*
        =H_C-x(\zeta)I_d+\frac{\zeta}{a}VV^*
        =F(\zeta^{-1}),
\]
which proves \eqref{eq:determinant-identity}.
\end{proof}

\begin{remark}[Coefficient matrix versus determinant]\label{rem:determinant-scope}
The function \(\det F\) determines \(\det S\), but it does not determine the
individual entries of \(S\).  Formula \eqref{eq:S-from-F} also requires the
attachment operator \(V\) with its cusp labels.  Moreover, reciprocal common
factors cancel from \(\det S\).  Factors arising from core-supported
eigenfunctions always cancel, as proved in Section~\ref{sec:cusp-invisible};
additional reciprocal coincidences among other blocks may cancel in the
determinant even when the matrix-valued scattering data remain nontrivial.
The finite-tail form of \eqref{eq:S-from-F} is classical in graph scattering
\cite[Eq.~(32)]{VarbanovBrun}; its role here is to supply the exactly normalized
coefficient matrix used in Theorem~\ref{thm:finite-resonance-comparison}.
\end{remark}

\subsection{Resolvent jump and spectral decomposition}\label{subsec:jump-resolution}

The explicit formula above also fixes the Plancherel measure.  For \(0<\theta<\pi\), put
\[
        \zeta=e^{i\theta},\qquad x=2a\cos\theta,
        \qquad \Psi_\theta=\Psi(e^{i\theta}).
\]
Thus
\begin{equation}\label{eq:Psi-explicit}
 \Psi_\theta\alpha
 =
 \left(
 (e^{-i\theta}-e^{i\theta})F(e^{i\theta})^{-1}V\alpha,\,
 \bigl(e^{in\theta}\alpha+e^{-in\theta}S(e^{i\theta})\alpha\bigr)_{n\ge1}
 \right).
\end{equation}
At the finite set \(\Theta_F\), assign arbitrary measurable values to
\(\Psi_\theta\); changing values on this null set does not affect any spectral
integral.

\begin{proposition}[Spectral type and multiplicity]\label{prop:spectral-type}
For \(J\) in \eqref{eq:block-J},
\[
        \Spec_{\mathrm{ess}}(J)=[-2a,2a].
\]
Its singular continuous spectrum is empty, its pure point subspace is finite-dimensional, and its absolutely continuous spectrum has multiplicity \(r\) almost everywhere on \((-2a,2a)\).
\end{proposition}

\begin{proof}
Decoupling the core and the \(r\) tails is a finite-rank perturbation.  Weyl's
theorem gives the essential spectrum, and the Kato--Rosenblum theorem
\cite{Kato} preserves the free absolutely continuous multiplicity \(r\)
almost everywhere.

It remains to exclude singular continuous spectrum.  Let \(\mathscr K\) be
the span of the entire core and the first vertex of each tail, let
\(\iota_{\mathscr K}:\mathscr K\hookrightarrow\mathscr H\) be the inclusion,
and let \(P_{\mathscr K}=\iota_{\mathscr K}^*\) be coordinate restriction.
Thus, for \(z\in\mathbb C\setminus\Spec(J)\), the compressed resolvent is the
type-correct operator
\[
 \iota_{\mathscr K}^*(J-z)^{-1}\iota_{\mathscr K}
 =P_{\mathscr K}(J-z)^{-1}P_{\mathscr K}^*:\mathscr K\to\mathscr K.
\]
If \(P_{\mathscr K}\) is instead regarded as the orthogonal projection on
\(\mathscr H\), the same compression is written
\(P_{\mathscr K}(J-z)^{-1}|_{\mathscr K}\).  This finite-dimensional
subspace is cyclic for \(J\): repeated application of \(J\) reaches every
finitely supported tail vector.  By the Schur complement in
Lemma~\ref{lem:exact-schur-complement}, every matrix entry of the compressed
resolvent is rational in \(\zeta\).  Its boundary values are therefore real
analytic on each compact subinterval of the open band after removal of
finitely many points.  Stone's formula shows that the spectral measures of
vectors in \(\mathscr K\) are absolutely continuous there; the omitted
finite set can support only atoms, not a singular continuous measure.
Cyclicity gives the same conclusion on \(\mathscr H\).  Finally, the poles
of the finitely many rational resolvent entries form a finite set, so the pure
point subspace is finite-dimensional.
\end{proof}

\begin{theorem}[Resolvent jump and Plancherel formula]\label{thm:resolvent-jump}
Let \(R(z)=(J-z)^{-1}\).  As an identity of sesquilinear forms on finitely supported vectors,
\begin{equation}\label{eq:resolvent-jump}
        R(x+i0)-R(x-i0)
        =\frac{i}{2a\sin\theta}\Psi_\theta\Psi_\theta^*,
        \qquad x=2a\cos\theta,
\end{equation}
for \(\theta\in(0,\pi)\setminus\Theta_F\).  Consequently
\begin{equation}\label{eq:standard-plancherel}
        P_{\mathrm{ac}}
        =\frac1{2\pi}\int_0^\pi
          \Psi_\theta\Psi_\theta^*\,d\theta.
\end{equation}
If \(\{\varphi_j\}_{j=1}^N\) is an orthonormal basis of the finite-dimensional pure point subspace, then the complete resolution is
\begin{equation}\label{eq:complete-standard-resolution}
        I_{\mathscr H}
        =\sum_{j=1}^N\varphi_j\varphi_j^*
         +\frac1{2\pi}\int_0^\pi
          \Psi_\theta\Psi_\theta^*\,d\theta.
\end{equation}
For a finitely supported \(f\), define
\[
        \widehat f_0(\theta)=\Psi_\theta^*f
        =\bigl(\langle f,\Psi_\theta e_c\rangle\bigr)_{c=1}^r,
\]
where the pairing is a finite sum.  Then
\begin{align}
        P_{\mathrm{ac}}f
        &=\frac1{2\pi}\int_0^\pi
          \Psi_\theta\widehat f_0(\theta)\,d\theta,
        \label{eq:standard-expansion}\\
        \|P_{\mathrm{ac}}f\|^2
        &=\frac1{2\pi}\int_0^\pi
          \|\widehat f_0(\theta)\|_{\CC^r}^2\,d\theta.
        \label{eq:standard-parseval}
\end{align}
In particular,
\[
        \|f\|^2
        =\sum_{j=1}^N|\langle f,\varphi_j\rangle|^2
         +\frac1{2\pi}\int_0^\pi
          \|\widehat f_0(\theta)\|_{\CC^r}^2\,d\theta.
\]
The integral in \eqref{eq:standard-expansion} is initially a weak integral;
its extension to arbitrary \(L^2\)-vectors is stated in Theorem
\ref{thm:unitary-scattering-transform} below.
\end{theorem}

\begin{proof}
Let \(J_D=H_C\oplus J_0^{(r)}\), let \(Q=J-J_D\), and write
\(R_D(z)=(J_D-z)^{-1}\), \(R_\pm=R(x\pm i0)\), and
\(R_{D,\pm}=R_D(x\pm i0)\).  For the Dirichlet half-line, the unit-incoming generalized eigenfunction is
\[
        e^{in\theta}-e^{-in\theta}=2i\sin(n\theta),
\]
and let \(\Psi_\theta^0\) be the direct sum of these \(r\) free states, extended by zero on the core.  Away from the eigenvalues of the decoupled finite core,
\[
 R_{D,+}-R_{D,-}
 =\frac{i}{2a\sin\theta}\Psi_\theta^0(\Psi_\theta^0)^*.
\]
The second resolvent identity gives the exact stationary factorization
\[
 R_+-R_-
 =(I+R_{D,+}Q)^{-1}(R_{D,+}-R_{D,-})
  (I+QR_{D,-})^{-1}.
\]
All boundary products are interpreted as limits of sesquilinear forms on finitely supported vectors.  Since \(Q\) has finite rank, the required inverses reduce to meromorphic finite matrices on \(\operatorname{ran}Q\).
Since \(R_{D,+}^*=R_{D,-}\), the two exterior factors are adjoints.  The
generalized-eigenfunction map
\[
        (I+R_{D,+}Q)^{-1}\Psi_\theta^0
\]
solves the coupled eigenvalue equation and has unit incoming amplitudes; uniqueness in Theorem~\ref{thm:exact-scattering} identifies it with \(\Psi_\theta\).  This proves \eqref{eq:resolvent-jump} for every block at once.

For comparison, the core--core block can be checked directly.  On the branch
\(|\zeta|>1\), the free tail Weyl function is \(-\zeta^{-1}/a\), so the two
core boundary resolvents are \(F(\zeta)^{-1}\) and
\(F(\zeta^{-1})^{-1}\), and
\[
\begin{aligned}
 F(\zeta)^{-1}-F(\zeta^{-1})^{-1}
 &=F(\zeta)^{-1}\{F(\zeta^{-1})-F(\zeta)\}F(\zeta^{-1})^{-1}\\
 &=\frac{\zeta-\zeta^{-1}}a
   F(\zeta)^{-1}VV^*F(\zeta^{-1})^{-1}.
\end{aligned}
\]
On \(\zeta=e^{i\theta}\), this agrees with the core block of the factorization above by \eqref{eq:core-poisson}.

Both sides of \eqref{eq:resolvent-jump} are meromorphic in \(\zeta\).
Consequently the identity extends across the auxiliary poles arising from
eigenvalues of the decoupled finite core whenever the coupled boundary
resolvent is regular.

Stone's formula for \(R(z)=(J-z)^{-1}\) is
\[
        dP_{\mathrm{ac}}(x)
        =\frac{1}{2\pi i}\{R(x+i0)-R(x-i0)\}\,dx.
\]
Since \(dx=-2a\sin\theta\,d\theta\) when the band is traversed from \(2a\) to \(-2a\), integrating the jump and adding the finitely many atoms from Proposition~\ref{prop:spectral-type} gives \eqref{eq:standard-plancherel}--\eqref{eq:complete-standard-resolution}.  The expansion and Parseval identities follow first for finitely supported vectors and then by density.
\end{proof}

\begin{theorem}[Unitary scattering transform]
\label{thm:unitary-scattering-transform}
Let \(\mathscr H_{\mathrm{ac}}=\operatorname{ran}P_{\mathrm{ac}}\).  The map
on finitely supported vectors given by
\[
 f\longmapsto \widehat f_0(\theta)=\Psi_\theta^*f
\]
depends only on \(P_{\mathrm{ac}}f\) as an almost-everywhere equivalence
class and extends uniquely to a unitary operator
\begin{equation}\label{eq:unitary-scattering-transform}
 \mathcal F_{\mathrm{ac}}:\mathscr H_{\mathrm{ac}}
 \longrightarrow
 L^2\!\left((0,\pi),\mathbb C^r;\frac{d\theta}{2\pi}\right).
\end{equation}
It diagonalizes \(J\):
\begin{equation}\label{eq:unitary-intertwining}
 (\mathcal F_{\mathrm{ac}}Jf)(\theta)
 =2a\cos\theta\,(\mathcal F_{\mathrm{ac}}f)(\theta)
\end{equation}
for \(f\in\mathscr H_{\mathrm{ac}}\).  For arbitrary
\(f\in\mathscr H\), the notation
\(\widehat f=\mathcal F_{\mathrm{ac}}P_{\mathrm{ac}}f\) denotes this
\(L^2\)-equivalence class, and
\begin{equation}\label{eq:unitary-scattering-inverse}
 P_{\mathrm{ac}}f
 =\frac1{2\pi}\int_0^\pi\Psi_\theta\widehat f(\theta)\,d\theta
\end{equation}
in the weak, equivalently \(L^2\), sense.
\end{theorem}

\begin{proof}
Equation \eqref{eq:standard-parseval} makes the transform an isometry on the
image under \(P_{\mathrm{ac}}\) of the finitely supported vectors, which is
dense in \(\mathscr H_{\mathrm{ac}}\), and proves that it vanishes on the pure
point component.  Hence it has a unique isometric extension.  The formal
eigenfunction equation and finite
summation give \eqref{eq:unitary-intertwining} first on a dense subspace and
then on all of \(\mathscr H_{\mathrm{ac}}\), since \(J\) is bounded.

The range is therefore a closed reducing subspace for multiplication by
\(2a\cos\theta\).  By Proposition~\ref{prop:spectral-type}, its spectral
multiplicity is \(r\) almost everywhere on the band.  Because
\(\theta\mapsto2a\cos\theta\) is one-to-one on \((0,\pi)\), every reducing
subspace of the target is described by a measurable family of subspaces of
\(\mathbb C^r\).  If the range were proper, the fiber dimension would be less
than \(r\) on a set of positive measure, contradicting the multiplicity
statement.  Thus the range is the whole target.  Formula
\eqref{eq:unitary-scattering-inverse} is the adjoint of this unitary map and
agrees with \eqref{eq:standard-expansion} on the dense subspace of finitely
supported vectors.
\end{proof}

\subsection{Arithmetic components and incoming normalization}
\label{subsec:automorphic-normalization}

\begin{setup}[Fixed-level components and cusp data]
\label{setup:global-automorphic}
For the statements involving classical Eisenstein series, let \(k\) be a
global function field with constant field \(\mathbb F_q\), let \(\infty\) be
a degree-one place, and identify \(k_\infty\) with \(K\).  Write
\(\mathbf G=\mathrm{PGL}_2\), and put
\[
 G_\infty=\mathbf G(k_\infty),
 \qquad G_f=\mathbf G(\mathbb A_{k,f}),
\]
where \(\mathbb A_k\) is the adele ring of \(k\) and
\(\mathbb A_{k,f}\) denotes its finite-adele factor.  Fix a compact open
subgroup \(K_f<G_f\).  Choose representatives
\(g_1,\ldots,g_h\in G_f\) for the finite class set
\(\mathbf G(k)\backslash G_f/K_f\), whose finiteness follows from
function-field reduction theory \cite{Harder,HarderAutomorphic}, and put
\begin{equation}\label{eq:adelic-component-lattices}
 \Gamma_j
 =\mathbf G(k)\cap g_jK_fg_j^{-1}
 <G_\infty.
\end{equation}
Here the intersection is formed using the diagonal copy of
\(\mathbf G(k)\), and the resulting group is viewed through its
\(\infty\)-component.  There are canonical decompositions
\begin{align}
 \mathbf G(k)\backslash\mathbf G(\mathbb A_k)/K_f
 &\cong\bigsqcup_{j=1}^h\Gamma_j\backslash G_\infty,
 \label{eq:adelic-component-decomposition}\\
 \mathbf G(k)\backslash\mathbf G(\mathbb A_k)/(K_f\mathcal K)
 &\cong\bigsqcup_{j=1}^h
 \Gamma_j\backslash G_\infty/\mathcal K.
 \label{eq:adelic-vertex-decomposition}
\end{align}
Normalize Haar measure on \(G_f\) by
\(\operatorname{vol}(K_f)=1\), and use the quotient measure induced from
this measure and the local Haar measure on \(G_\infty\) fixed in
Setup~\ref{setup:quotient-adjacency}.  With these conventions, the
decompositions above are measure preserving on every component; in
particular, no component-dependent scalar is suppressed.

Let \(\varpi=t^{-1}\) under the chosen identification
\(k_\infty\cong\mathbb F_q(\!(t^{-1})\!)\).  We use the
\emph{unnormalized spherical Hecke operator at \(\infty\)}
\begin{equation}\label{eq:infinity-hecke-operator}
 (\mathsf T_\infty F)(g)
 =\int_{\mathcal K\diag(\varpi,1)\mathcal K}F(gh)\,dh,
 \qquad \operatorname{vol}(\mathcal K)=1.
\end{equation}
The double coset is the disjoint union of \(q+1\) right
\(\mathcal K\)-cosets.  Hence, on right \(\mathcal K\)-invariant functions,
\(\mathsf T_\infty\) is exactly the \((q+1)\)-neighbor adjacency operator
\eqref{eq:quotient-adjacency}, rather than its probability-normalized version.

Fix one component \(j_0\), set \(\Gamma=\Gamma_{j_0}\), and let
\(Y=\Gamma\backslash\mathcal T\).  Write \(r\) for the number of cusps of
this component.  All identities below are componentwise.
The measure on this component is induced from the local Haar measure on
\(G_\infty\) normalized by \(\operatorname{vol}(\mathcal K)=1\), as in
Setup~\ref{setup:quotient-adjacency}; no unmentioned Tamagawa or finite-level
scalar is absorbed into \(\lambda(v)=|\Gamma_v|^{-1}\).

For every cusp \(c\) of \(Y\), choose its endpoint
\(\xi_c\in\mathbb P^1(k)\), let
\(P_c=\operatorname{Stab}_{\mathbf G}(\xi_c)\) be the corresponding
\(k\)-parabolic, and let \(N_c\) be its unipotent radical.  Put
\[
 \Gamma_c=\Gamma\cap P_c(k),
 \qquad \Gamma_{N,c}=\Gamma\cap N_c(k).
\]
Let \(\alpha_c:P_c\to\mathbb G_m\) be the root character: after a
\(k\)-rational conjugation carrying \(\xi_c\) to \([1:0]\), it sends
the class of \(\left(\begin{smallmatrix}a&*\\0&d\end{smallmatrix}\right)\)
to \(a/d\).  Normalize the valuation by
\(v_\infty(t^{-1})=1\), and choose the integer-valued Busemann height \(h_c\),
increasing towards \(\xi_c\), so that
\begin{equation}\label{eq:busemann-character-root}
 h_c(pg\mathcal K)-h_c(g\mathcal K)
 =\nu_c(p)=-v_\infty\!\left(\alpha_c(p)\right)
 \qquad(p\in P_c(K)).
\end{equation}

By rank-one reduction theory, the cusp may be cut so that its remaining
vertices have representatives \(g_{c,n}\mathcal K\), \(n\ge1\), with
\(h_c(g_{c,n}\mathcal K)=n\).  Put
\(\mathscr H_c=\{x\in V(\mathcal T):h_c(x)\ge1\}\); its level sets
\(h_c^{-1}(n)\) are the corresponding horospheres.  We choose the cut
sufficiently deep that distinct \(\Gamma\)-translates of the retained
horoballs \(\mathscr H_c\) are disjoint, the stabilizer of \(\mathscr H_c\) is
\(\Gamma_c\), and
\begin{equation}\label{eq:one-vertex-per-horosphere}
 \Gamma_c\backslash h_c^{-1}(n)
 \quad\text{maps to the single quotient vertex }(c,n)
 \quad(n\ge1).
\end{equation}
These are the only
reduction-theoretic properties used below
\cite{Harder,HarderAutomorphic,Serre}.
\end{setup}

\begin{lemma}[Arithmetic cusp stabilizers preserve height]
\label{lem:arithmetic-cusp-height}
For every \(\gamma_c\in\Gamma_c\), one has
\(\nu_c(\gamma_c)=0\).  Consequently \(h_c\) is
\(\Gamma_c\)-invariant and
\begin{equation}\label{eq:explicit-spherical-flat-section}
 f_{c,s}(g)=q^{s h_c(g\mathcal K)}
\end{equation}
is a well-defined left-\(\Gamma_c\)-invariant, right-\(\mathcal K\)-invariant
spherical flat section.  More explicitly,
\begin{equation}\label{eq:spherical-flat-section-equivariance}
 f_{c,s}(pgk)=q^{s\nu_c(p)}f_{c,s}(g)
 \qquad(p\in P_c(K),\ k\in\mathcal K).
\end{equation}
\end{lemma}

\begin{proof}
The finite component of \(\gamma_c\) lies in the compact group
\(g_{j_0}K_fg_{j_0}^{-1}\) and in \(P_c(\mathbb A_{k,f})\).  Hence
\(\alpha_c(\gamma_c)\) lies, at every finite place \(v\), in a compact
subgroup of \(k_v^\times\); its valuation is therefore zero.  Since
\(\alpha_c(\gamma_c)\in k^\times\), the product formula and
\(\deg(\infty)=1\) give
\[
 v_\infty\!\left(\alpha_c(\gamma_c)\right)
 =-\sum_{v\ne\infty}\deg(v)
   v\!\left(\alpha_c(\gamma_c)\right)=0.
\]
Equation~\eqref{eq:busemann-character-root} proves the claim and the stated
invariance properties.
\end{proof}

Normalize Haar measure \(du_c\) on \(N_c(K)\) by
\[
 \operatorname{vol}\bigl(\Gamma_{N,c}\backslash N_c(K),du_c\bigr)=1,
\]
and, for a left-\(\Gamma\)-invariant function \(\varphi\), define its
componentwise unipotent constant term by
\begin{equation}\label{eq:component-unipotent-constant-term}
 \operatorname{CT}_c\varphi(g)
 =\int_{\Gamma_{N,c}\backslash N_c(K)}\varphi(ug)\,du_c.
\end{equation}
The quotient in this integral is compact for the arithmetic cusp under
consideration.

\begin{lemma}[A deep ray is its normalized constant term]
\label{lem:ray-equals-constant-term}
If \(\varphi\) is left \(\Gamma\)-invariant and right
\(\mathcal K\)-invariant, then, for every retained cusp vertex,
\begin{equation}\label{eq:ray-constant-term-identification}
 \operatorname{CT}_c\varphi(g_{c,n})=\varphi(g_{c,n}),
 \qquad n\ge1.
\end{equation}
Thus no horospherical volume is hidden in the passage from automorphic
constant terms to functions on the quotient ray.
\end{lemma}

\begin{proof}
The group \(N_c(K)\) preserves each horosphere based at \(\xi_c\).  By the
choice of the cut, every vertex of the horosphere through
\(g_{c,n}\mathcal K\) maps to the single quotient vertex \((c,n)\).
Therefore \(\varphi(ug_{c,n})=\varphi(g_{c,n})\) for every \(u\in N_c(K)\).
The normalization of \(du_c\) in \eqref{eq:component-unipotent-constant-term}
now gives \eqref{eq:ray-constant-term-identification}.
\end{proof}

For \(\Re s>1\), define the \emph{raw} lattice Poincar\'e series
\begin{equation}\label{eq:classical-component-Eisenstein-sum}
 E_c^0(g,s)
 =\sum_{\gamma\in\Gamma_c\backslash\Gamma}
   f_{c,s}(\gamma g),
 \qquad g\in G_\infty.
\end{equation}
The series converges locally uniformly and absolutely in this half-plane.  By
rank-one Eisenstein theory over function fields, it and its constant terms
continue meromorphically in \(s\); in the notation of
\cite{LiEisenstein}, convergence, the constant-term formula, and the
functional equation are given respectively by Theorem~2.3, Theorem~3.1, and
Theorem~5.2.
For a cusp \(d\), write
\(E_c^0(d,n;s)=E_c^0(g_{d,n},s)\).

\begin{lemma}[Incoming normalization of the raw Poincar\'e series]
\label{lem:raw-eisenstein-normalization}
For \(\Re s>1\), the series \(E_c^0(\,\cdot\,,s)\) is an adjacency
eigenfunction with
eigenvalue \(q^s+q^{1-s}\).  For the cusp representatives and height origins
fixed in Setup~\ref{setup:global-automorphic}, its deep-cusp expansion is
\begin{equation}\label{eq:classical-component-constant-term}
 E_c^0(d,n;s)
 =\delta_{dc}q^{ns}+M_{dc}(s)q^{n(1-s)}.
\end{equation}
Here \(\delta_{dc}\) is the Kronecker delta, and the coefficients are assembled
into the classical constant-term matrix
\begin{equation}\label{eq:classical-constant-term-matrix-defined}
 M(s)=\bigl(M_{dc}(s)\bigr)_{1\le d,c\le r}.
\end{equation}
Both the eigenvalue equation and the expansion continue meromorphically in
\(s\), with coincident radial modes interpreted by continuation.
In particular, the incoming matrix of the raw column family
\(E_s^0=(E_1^0,\ldots,E_r^0)\) is exactly \(I_r\).  Consequently no
additional scalar is present in the classical constant-term normalization
used in this paper, and we set
\begin{equation}\label{eq:classical-family-defined}
 E_s^{\mathrm{cl}}=E_s^0.
\end{equation}
\end{lemma}

\begin{proof}
At every vertex, one neighbor raises \(h_c\) by one and \(q\) neighbors lower
it by one.  Hence
\[
 Af_{c,s}=(q^s+q^{1-s})f_{c,s}.
\]
Absolute convergence for \(\Re s>1\) permits termwise application of \(A\),
which proves the eigenvalue equation there.

We compute the value on a retained \(d\)-ray through its constant term.  By
Lemma~\ref{lem:ray-equals-constant-term} and absolute convergence,
\begin{equation}\label{eq:raw-unfolding-start}
 E_c^0(d,n;s)
 =\int_{\Gamma_{N,d}\backslash N_d(K)}
   \sum_{\gamma\in\Gamma_c\backslash\Gamma}
   f_{c,s}(\gamma u g_{d,n})\,du_d,
\end{equation}
and the sum and integral may be interchanged: the unipotent quotient is
compact and the series converges locally uniformly there.  Fix a rational
minimal parabolic \(P<\mathbf G\).  Choose
\(\sigma_c,\sigma_d\in\mathbf G(k)\) such that
\[
 P_c=\sigma_cP\sigma_c^{-1},
 \qquad P_d=\sigma_dP\sigma_d^{-1},
\]
and let \(w\in\mathbf G(k)\) represent the nontrivial Weyl-group element.
The relative position of the summand indexed by \(\gamma\) is the double
coset of \(\sigma_c^{-1}\gamma\sigma_d\), and the rational Bruhat
decomposition is
\[
 P(k)\backslash\mathbf G(k)/P(k)
 =\{P(k),P(k)wP(k)\}.
\]
The identity cell condition
\(\sigma_c^{-1}\gamma\sigma_d\in P(k)\) is equivalent to
\(\gamma\xi_d=\xi_c\).  Since \(\gamma\in\Gamma\), this occurs precisely
when \(c=d\).  Taking \(\sigma_c=\sigma_d\) in that case, the elements in
the identity cell are exactly \(\Gamma\cap P_c(k)=\Gamma_c\), so they give
one element of \(\Gamma_c\backslash\Gamma\).  The normalization
\(\operatorname{vol}(\Gamma_{N,c}\backslash N_c(K))=1\) gives
\[
 \int_{\Gamma_{N,c}\backslash N_c(K)}
 f_{c,s}(u g_{c,n})\,du_c
 =f_{c,s}(g_{c,n})=q^{ns}.
\]
For the nonidentity cell, the unfolded sum and integral are the rank-one
standard intertwining term---that is, the open-Bruhat-cell unipotent
integral---applied to the \(c\)-cusp section.  This operator carries the
\(s\)-equivariance in
\eqref{eq:spherical-flat-section-equivariance} to \((1-s)\)-equivariance.
Since \(h_d(g_{d,n}\mathcal K)=n\), its value therefore has the form
\[
 M_{dc}(s)q^{n(1-s)},
\]
where \(M_{dc}(s)\) is independent of \(n\).  This is the nontrivial-Weyl-cell
term in the standard rank-one constant-term formula
\cite[Thm.~3.1]{LiEisenstein}; the disjoint Bruhat decomposition shows that it
cannot contribute a second \(q^{ns}\) term.  As a check, the eigenvalue
equation on the homogeneous \(d\)-ray has characteristic solutions
\(q^{ns}\) and \(q^{n(1-s)}\), so these are the only possible radial
dependences when they are distinct.  We obtain
\eqref{eq:classical-component-constant-term} for \(\Re s>1\).  The
meromorphic continuation recalled above proves the stated identities in
general; at parameters satisfying \(q^{2s-1}=1\), the coincident-root formula
is understood as this meromorphic continuation.  This is the cusp-coordinate
form of the usual constant-term calculation; compare
\cite[Chaps.~6--7]{LanglandsEisenstein}.
\end{proof}

\begin{remark}[Conversion from another incoming normalization]
\label{rem:incoming-matrix-normalization}
If a differently normalized classical or adelic column family
\(\widetilde E_s\) has incoming and outgoing matrices
\(A_{\mathrm{in}}(s)\) and \(B_{\mathrm{out}}(s)\), then ordinary right
multiplication of the column family gives, wherever
\(A_{\mathrm{in}}(s)\) is invertible,
\begin{equation}\label{eq:incoming-matrix-normalization}
 E_s^{\mathrm{cl}}=\widetilde E_sA_{\mathrm{in}}(s)^{-1},
 \qquad
 M(s)=B_{\mathrm{out}}(s)A_{\mathrm{in}}(s)^{-1}.
\end{equation}
Both identities extend meromorphically.  Thus
\eqref{eq:classical-family-defined} is the special case
\(A_{\mathrm{in}}(s)=I_r\), rather than an implicit convention.
\end{remark}

\subsection{Fixed-level adelic assembly}
\label{subsec:adelic-assembly}

Fix a rational minimal parabolic \(P<\mathbf G\), its fixed endpoint
\(\xi_P\in\mathbb P^1(k)\), its root character
\(\alpha:P\to\mathbb G_m\), and an integer-valued Busemann function
\(h_P:G_\infty/\mathcal K\to\mathbb Z\) satisfying
\begin{equation}\label{eq:standard-parabolic-height}
 h_P(p g\mathcal K)-h_P(g\mathcal K)
 =-v_\infty(\alpha(p))
 \qquad(p\in P(k),\ g\in G_\infty).
\end{equation}
For every component \(j\), repeat the preceding choices of cusp heights and
write
\[
 \mathscr C_j=\Gamma_j\backslash\mathbf G(k)/P(k),
 \qquad r_j=|\mathscr C_j|.
\]
For each \(c\in\mathscr C_j\), choose \(\gamma_{j,c}\in\mathbf G(k)\) with
\[
 c=\Gamma_j\gamma_{j,c}P(k),
 \qquad
 \xi_{j,c}=\gamma_{j,c}\xi_P,
 \qquad
 P_{j,c}=\gamma_{j,c}P\gamma_{j,c}^{-1}.
\]
Denote the resulting height, flat section, and raw Poincar\'e series by
\(h_{j,c}\), \(f_{j,c,s}\), and \(E_{j,c}^0\), respectively, and put
\begin{equation}\label{eq:component-column-family}
 E_{j,s}^{\mathrm{cl}}
 =(E_{j,c}^0)_{c\in\mathscr C_j}.
\end{equation}
Reduction theory makes these sets finite.  The complete fixed-level cusp set
is
\begin{equation}\label{eq:adelic-cusp-set}
 \mathscr C(K_f)=P(k)\backslash G_f/K_f.
\end{equation}
The finite component decomposition shows that every
\(\mathfrak c\in\mathscr C(K_f)\) is represented by at least one pair
\((j,c)\) through the map in
\eqref{eq:adelic-cusp-parametrization}; Proposition
\ref{prop:adelic-component-compatibility} proves that this pair is unique.
Choose that pair, set \(\gamma_{\mathfrak c}=\gamma_{j,c}\), and put
\[
 x_{\mathfrak c}=\gamma_{\mathfrak c}^{-1}g_j\in G_f.
\]
Set
\[
 \Gamma_{j,c}
 =\Gamma_j\cap
  \gamma_{\mathfrak c}P(k)\gamma_{\mathfrak c}^{-1},
\]
the stabilizer in \(\Gamma_j\) of the selected cusp.
Write \(h_c=h_{j,c}\) for the selected height at this cusp.  There is an integer
\(\kappa_{\mathfrak c}\) such that
\begin{equation}\label{eq:component-global-height-shift}
 h_c(g\mathcal K)
 =h_P(\gamma_{\mathfrak c}^{-1}g\mathcal K)
  +\kappa_{\mathfrak c}.
\end{equation}
Indeed, both sides before addition of the constant are integer-valued
Busemann functions based at \(\xi_{j,c}\), so their difference is constant;
this also fixes precisely the dependence on the chosen height origin.
Define a right \(K_f\mathcal K\)-invariant, left \(P(k)\)-invariant
cusp-coordinate flat section, for
\(g=(g_\infty,g_f)\in\mathbf G(\mathbb A_k)=G_\infty\times G_f\), by
\begin{equation}\label{eq:adelic-cusp-coordinate-section}
 \phi_{\mathfrak c,s}(g_\infty,g_f)
 =
 \begin{cases}
 q^{s\{h_P(p^{-1}g_\infty\mathcal K)+\kappa_{\mathfrak c}\}},
   &g_f=p x_{\mathfrak c}k_f,
     \quad p\in P(k),\ k_f\in K_f,\\[2mm]
 0,&g_f\notin P(k)x_{\mathfrak c}K_f.
 \end{cases}
\end{equation}
In this formula the same symbol \(p\) denotes the finite and infinite local
components of its diagonal adelic image.  To check independence of the
factorization, suppose
\(g_f=p x_{\mathfrak c}k_f=p'x_{\mathfrak c}k_f'\).  Then
\[
 a=p'^{-1}p\in
 P(k)\cap x_{\mathfrak c}K_fx_{\mathfrak c}^{-1}.
\]
The element \(\alpha(a)\in k^\times\) has valuation zero at every finite
place and hence, by the product formula and \(\deg(\infty)=1\), also at
\(\infty\).  Equation~\eqref{eq:standard-parabolic-height} therefore gives
\[
 h_P(p^{-1}g_\infty\mathcal K)
 =h_P(p'^{-1}g_\infty\mathcal K).
\]
Replacing \(g\) by \(rg\), with \(r\in P(k)\), replaces \(p\) by \(rp\)
and leaves \((rp)^{-1}rg_\infty=p^{-1}g_\infty\); right
\(K_f\mathcal K\)-invariance is immediate.  Thus the claimed invariance and
well-definedness hold.  The sections in
\eqref{eq:adelic-cusp-coordinate-section} have disjoint finite supports and
form a basis of the finite-dimensional cusp-coordinate inducing space
\begin{equation}\label{eq:adelic-inducing-space}
 \mathcal I_s(K_f)
 =\operatorname{span}\{\phi_{\mathfrak c,s}:
     \mathfrak c\in\mathscr C(K_f)\},
 \qquad
 \dim\mathcal I_s(K_f)=\sum_{j=1}^h r_j.
\end{equation}
For \(\Re s>1\), define
\begin{equation}\label{eq:adelic-cusp-eisenstein-series}
 E_{\mathfrak c}^{\mathrm{ad}}(g,s)
 =\sum_{\delta\in P(k)\backslash\mathbf G(k)}
   \phi_{\mathfrak c,s}(\delta g),
 \qquad g\in\mathbf G(\mathbb A_k),
 \qquad
 E_s^{\mathrm{ad}}
 =(E_{\mathfrak c}^{\mathrm{ad}})_{
      \mathfrak c\in\mathscr C(K_f)}.
\end{equation}
For now these series are defined in the absolute-convergence half-plane;
Proposition~\ref{prop:adelic-component-compatibility} identifies all their
component restrictions and thereby supplies their meromorphic continuation.

\begin{proposition}[Compatibility with the fixed-level adelic family]
\label{prop:adelic-component-compatibility}
The map
\begin{equation}\label{eq:adelic-cusp-parametrization}
 \bigsqcup_{j=1}^h
 \Gamma_j\backslash\mathbf G(k)/P(k)
 \longrightarrow P(k)\backslash G_f/K_f,
 \qquad
 \Gamma_j\gamma P(k)\longmapsto
 P(k)\gamma^{-1}g_jK_f,
\end{equation}
is a bijection.  If \(\mathfrak c\) corresponds to
\(c=\Gamma_j\gamma_{\mathfrak c}P(k)\), then, for \(\Re s>1\) and
\(1\le\ell\le h\),
\begin{equation}\label{eq:adelic-column-component-restriction}
 E_{\mathfrak c}^{\mathrm{ad}}(g_\infty,g_\ell;s)
 =
 \begin{cases}
 E_{j,c}^0(g_\infty,s),&\ell=j,\\
 0,&\ell\ne j,
 \end{cases}
\end{equation}
where \(E_{j,c}^0\) is the raw Poincar\'e series
\eqref{eq:classical-component-Eisenstein-sum} for the \(c\)-th cusp of
\(\Gamma_j\).  Consequently, under the measure-preserving decomposition
\eqref{eq:adelic-component-decomposition},
\begin{equation}\label{eq:adelic-component-eisenstein-compatibility}
 E_s^{\mathrm{ad}}
 \longleftrightarrow
 \bigoplus_{j=1}^h E_{j,s}^{\mathrm{cl}}
\end{equation}
as meromorphic column families.  In the componentwise volume-one
constant-term convention of
\eqref{eq:component-unipotent-constant-term}, their full incoming matrix is
\(\bigoplus_{j=1}^h I_{r_j}=I_{\sum_jr_j}\), and \(\mathsf T_\infty\)
corresponds exactly to the direct sum of the component adjacency operators.
\end{proposition}

\begin{proof}
We first verify the asserted bijection.  If
\(\gamma'=\eta\gamma p\), with \(\eta\in\Gamma_j\) and \(p\in P(k)\), then
\(\eta^{-1}g_j=g_jk\) for some \(k\in K_f\), and hence
\[
 P(k)\gamma'^{-1}g_jK_f
 =P(k)p^{-1}\gamma^{-1}\eta^{-1}g_jK_f
 =P(k)\gamma^{-1}g_jK_f.
\]
Thus the map is well defined.  Given \(x\in G_f\), the finite component
decomposition supplies \(j\), \(\delta\in\mathbf G(k)\), and \(k\in K_f\)
with \(x=\delta g_jk\); the class
\(\Gamma_j\delta^{-1}P(k)\) maps to \(P(k)xK_f\), proving surjectivity.
Finally, suppose the images of \(\Gamma_j\gamma P(k)\) and
\(\Gamma_\ell\gamma'P(k)\) agree.  Then
\[
 \gamma^{-1}g_j=p\gamma'^{-1}g_\ell k
 \qquad(p\in P(k),\ k\in K_f).
\]
The defining double-coset representatives first give \(j=\ell\).  With
\(\eta=\gamma p\gamma'^{-1}\), the displayed equality becomes
\(g_j=\eta g_jk\), so \(\eta\in\Gamma_j\), and
\(\gamma=\eta\gamma'p^{-1}\).  The two source classes are therefore equal,
which proves injectivity.

We next prove the component formula.  Fix
\(x_{\mathfrak c}=\gamma_{\mathfrak c}^{-1}g_j\).  A summand in
\(E_{\mathfrak c}^{\mathrm{ad}}(g_\infty,g_\ell;s)\) is nonzero precisely
when
\[
 P(k)\delta g_\ell K_f
 =P(k)\gamma_{\mathfrak c}^{-1}g_jK_f.
\]
This is impossible for \(\ell\ne j\), by the choice of the distinct
\(\mathbf G(k)\backslash G_f/K_f\)-representatives.  For \(\ell=j\), the
surviving cosets are
\[
 P(k)\gamma_{\mathfrak c}^{-1}\eta,
 \qquad
 \eta\in
 (\Gamma_j\cap\gamma_{\mathfrak c}P(k)
       \gamma_{\mathfrak c}^{-1})\backslash\Gamma_j
 =\Gamma_{j,c}\backslash\Gamma_j.
\]
Since \(\eta g_j=g_jk_\eta\) for some \(k_\eta\in K_f\), equations
\eqref{eq:adelic-cusp-coordinate-section} and
\eqref{eq:component-global-height-shift} give
\[
 \phi_{\mathfrak c,s}
 \bigl(\gamma_{\mathfrak c}^{-1}\eta
       (g_\infty,g_j)\bigr)
 =q^{s h_c(\eta g_\infty\mathcal K)}
 =f_{j,c,s}(\eta g_\infty).
\]
Summing over \(\Gamma_{j,c}\backslash\Gamma_j\) proves
\eqref{eq:adelic-column-component-restriction}.  This is a pointwise
identity, so no finite Haar factor occurs.  Lemma
\ref{lem:raw-eisenstein-normalization} gives the identity incoming matrix,
and the measure normalization in Setup~\ref{setup:global-automorphic} makes
the Hilbert-space direct sum orthogonal.  Formula
\eqref{eq:infinity-hecke-operator} proves the Hecke--adjacency statement.
Finally, every \(E_{j,c}^0\) has the meromorphic continuation recalled before
Lemma~\ref{lem:raw-eisenstein-normalization}.  Formula
\eqref{eq:adelic-column-component-restriction} therefore defines the
meromorphic continuation of \(E_{\mathfrak c}^{\mathrm{ad}}\) on every
component, and the proved equality persists after continuation.
\end{proof}

\begin{remark}[Cusp and character coordinates]
Here a \emph{factorizable section} means a restricted tensor product of local
inducing sections.  The basis \(\{\phi_{\mathfrak c,s}\}\) is localized at
cusps and is generally not itself factorizable.  If \(\mathscr C(K_f)\) is
identified with a finite abelian
class group, as in the four-cusp elliptic example, the basis obtained from
unramified class-group characters is its finite Fourier transform.  Thus one
factorizable character section generally produces a linear combination of
cusp columns, whereas the complete character family and the complete cusp
family span the same fixed-level inducing space.  This is why
Proposition~\ref{prop:adelic-component-compatibility} concerns the full
family, rather than an individual factorizable column.
\end{remark}

\subsection{Automorphic normalization and the unitary Eisenstein transform}
\label{subsec:automorphic-unitary-transform}

We now translate \eqref{eq:S-from-F} and \eqref{eq:standard-plancherel} back
to the stabilizer-weighted quotient, denoted by \(Y\).  In this subsection
\(P_{\mathrm{ac}}\) refers to the absolutely continuous projection of \(A\)
on \(L^2(Y,\lambda)\); it is unitarily conjugate to the projection for \(J\)
in \eqref{eq:standard-plancherel}.  On the \(c\)-th cusp assume
\[
        \lambda(c,n)=w_c(q+1)q^{-n},
        \qquad
        W=\diag(w_1,\ldots,w_r).
\]
Set \(\zeta=q^{s-1/2}\), so the adjacency eigenvalue is
\(q^s+q^{1-s}=x(\zeta)\).  For an incoming coefficient vector
\(b\in\CC^r\), use Efrat's holomorphically normalized cusp modes
\begin{align}
        \psi_s^{\mathrm{in}}(n)
        &=q^{ns}(q^{s-1}-q^{1-s}),\\
        \psi_s^{\mathrm{out}}(n)
        &=q^{n(1-s)}(q^s-q^{-s}),
\end{align}
and denote by \(E_s^{\mathrm{hol}}(b)\) the generalized eigenfunction with cusp expansion
\begin{equation}\label{eq:automorphic-asymptotic}
        E_s^{\mathrm{hol}}(b)(c,n)
        =b_c\psi_s^{\mathrm{in}}(n)
        +(\Phi(s)b)_c\psi_s^{\mathrm{out}}(n).
\end{equation}
Away from the discrete singular set of \(F\), existence and uniqueness follow
from Theorem~\ref{thm:exact-scattering} after applying \(U^{-1}\); elsewhere
this notation refers to the meromorphic continuation.
Thus \(\Phi(s)\) is the coefficient matrix in the
\(\psi^{\mathrm{in/out}}\)-basis.  It is a holomorphically rescaled version of
the classical constant-term matrix in
\eqref{eq:classical-component-constant-term}.
Define
\begin{equation}\label{eq:cin-cout}
        c_{\mathrm{in}}(\zeta)=\frac{\zeta}{a}-a\zeta^{-1},
        \qquad
        c_{\mathrm{out}}(\zeta)=a\zeta-\frac{\zeta^{-1}}a.
\end{equation}

\begin{theorem}[Normalization dictionary]\label{thm:normalization-dictionary}
Under \(Uf=\lambda^{1/2}f\), the automorphic incoming and outgoing amplitudes in \eqref{eq:automorphic-asymptotic} become
\[
        K_{\mathrm{in}}(\zeta)b,
        \qquad
        K_{\mathrm{out}}(\zeta)\Phi(s)b,
\]
where
\[
        K_{\mathrm{in}}(\zeta)
        =\sqrt{q+1}\,c_{\mathrm{in}}(\zeta)W^{1/2},
        \qquad
        K_{\mathrm{out}}(\zeta)
        =\sqrt{q+1}\,c_{\mathrm{out}}(\zeta)W^{1/2}.
\]
Consequently
\begin{equation}\label{eq:Phi-S-dictionary}
        \Phi(s)
        =\frac{c_{\mathrm{in}}(\zeta)}{c_{\mathrm{out}}(\zeta)}
          W^{-1/2}S(\zeta)W^{1/2}.
\end{equation}
Away from the zeros of \(c_{\mathrm{in}}\), define the
constant-term-normalized Eisenstein family
\begin{equation}\label{eq:constant-term-Eisenstein-family}
        \mathcal E_s=E_s^{\mathrm{hol}}\,c_{\mathrm{in}}(\zeta)^{-1}.
\end{equation}
Its cusp expansion is
\begin{equation}\label{eq:normalized-constant-term}
        \mathcal E_s(b)(c,n)
        =q^{ns}b_c
         +q^{n(1-s)}\bigl(\mathcal S_{\mathrm{aut}}(s)b\bigr)_c,
\end{equation}
where the classical scattering matrix defined by Eisenstein constant terms is
\begin{equation}\label{eq:automorphic-scattering-dictionary}
        \mathcal S_{\mathrm{aut}}(s)
        =\frac{c_{\mathrm{out}}(\zeta)}{c_{\mathrm{in}}(\zeta)}\Phi(s)
        =W^{-1/2}S(\zeta)W^{1/2}.
\end{equation}
Equations \eqref{eq:constant-term-Eisenstein-family}--
\eqref{eq:automorphic-scattering-dictionary} continue meromorphically across
the singular parameters.  At every point of the open unitary axis
\(s=1/2+i\theta/\log q\), \(0<\theta<\pi\), at which the continued matrices
are holomorphic, both \(\Phi(s)\) and
\(\mathcal S_{\mathrm{aut}}(s)\) are unitary for the cusp-width inner product
\[
        \langle b_1,b_2\rangle_W
        =\sum_{c=1}^r w_c(b_1)_c\overline{(b_2)_c}
\]
and satisfy
\[
        \Phi(1-s)\Phi(s)=I_r,
        \qquad
        \mathcal S_{\mathrm{aut}}(1-s)\mathcal S_{\mathrm{aut}}(s)=I_r.
\]
At a removable singularity, the same assertions hold after taking the
holomorphic continuation.
\end{theorem}

\begin{proof}
Since \(q^{ns}=q^{n/2}\zeta^n\) and \(q^{n(1-s)}=q^{n/2}\zeta^{-n}\),
\[
\begin{aligned}
 U\psi_s^{\mathrm{in}}(c,n)
 &=\sqrt{(q+1)w_c}\,c_{\mathrm{in}}(\zeta)\zeta^n,\\
 U\psi_s^{\mathrm{out}}(c,n)
 &=\sqrt{(q+1)w_c}\,c_{\mathrm{out}}(\zeta)\zeta^{-n}.
\end{aligned}
\]
The standard outgoing-amplitude relation is therefore
\[
        K_{\mathrm{out}}(\zeta)\Phi(s)
        =S(\zeta)K_{\mathrm{in}}(\zeta),
\]
which is \eqref{eq:Phi-S-dictionary}.  On \(|\zeta|=1\),
\[
        |c_{\mathrm{in}}(\zeta)|
        =|c_{\mathrm{out}}(\zeta)|.
\]
The constant-term expansion and
\eqref{eq:automorphic-scattering-dictionary} follow by dividing the incoming
coefficient by \(c_{\mathrm{in}}\).  Thus \eqref{eq:S-unitary} is equivalent
to \(W\)-unitarity of \(\mathcal S_{\mathrm{aut}}\); since
\(|c_{\mathrm{in}}|=|c_{\mathrm{out}}|\), it is also equivalent to
\(W\)-unitarity of \(\Phi\).  Finally,
\[
        c_{\mathrm{in}}(\zeta^{-1})=-c_{\mathrm{out}}(\zeta),
        \qquad
        c_{\mathrm{out}}(\zeta^{-1})=-c_{\mathrm{in}}(\zeta),
\]
and \eqref{eq:S-functional} gives both automorphic functional equations.
\end{proof}

\begin{theorem}[Automorphic matrix Maass--Selberg relation]
\label{thm:automorphic-maass-selberg}
Let \(\chi_N=U^{-1}P_NU\), equivalently the coordinate projection retaining
the finite core and the first \(N\) vertices of every homogeneous cusp, where
\(N\ge1\).  Put
\[
 s_\theta=\frac12+\frac{i\theta}{\log q},
 \qquad
 \Sigma_\theta=\mathcal S_{\mathrm{aut}}(s_\theta),
 \qquad 0<\theta<\pi,
\]
away from \(\Theta_F\).  Then the constant-term-normalized spherical
Eisenstein family satisfies
\begin{equation}\label{eq:automorphic-maass-selberg}
\begin{aligned}
 \mathcal E_{s_\theta}^*\chi_N\mathcal E_{s_\theta}
 =(q+1)\Bigg[&(2N+1)W+i\Sigma_\theta^*W\dot\Sigma_\theta\\
 &+\frac{i}{2\sin\theta}
 \left(e^{-i(2N+1)\theta}W\Sigma_\theta
       -e^{i(2N+1)\theta}\Sigma_\theta^*W\right)\Bigg],
\end{aligned}
\end{equation}
where \(\dot\Sigma_\theta=\partial_\theta\Sigma_\theta\).
The left-hand side is the \(r\times r\) Gram matrix with entries
\[
 \bigl(\mathcal E_{s_\theta}^*\chi_N
 \mathcal E_{s_\theta}\bigr)_{cd}
 =\bigl\langle\chi_N\mathcal E_{s_\theta}(e_d),
   \mathcal E_{s_\theta}(e_c)\bigr\rangle;
\]
every such pairing is a finite sum.
In particular,
\(i\Sigma_\theta^*W\dot\Sigma_\theta\) is Hermitian.  After conjugation by
\(W^{-1/2}\), this middle term is exactly the channel matrix
\(\mathcal M(\theta)=iS_\theta^*\dot S_\theta\) from
\eqref{eq:channel-mass-matrix}.
\end{theorem}

\begin{proof}
Equation \eqref{eq:normalized-constant-term} and the definition of \(U\)
give the exact relation
\[
 U\mathcal E_{s_\theta}
 =\Psi_\theta\sqrt{q+1}\,W^{1/2}.
\]
Conjugating \eqref{eq:matrix-maass-selberg} by
\(\sqrt{q+1}\,W^{1/2}\), and using
\[
 S_\theta=W^{1/2}\Sigma_\theta W^{-1/2},
\]
gives \eqref{eq:automorphic-maass-selberg}.  The Hermitian assertion follows
either from the displayed identity or by differentiating
\(\Sigma_\theta^*W\Sigma_\theta=W\).
\end{proof}

\begin{corollary}[Automorphic recovery beyond the scattering determinant]
\label{cor:automorphic-matrix-consequence}
Let \(s_\theta=1/2+i\theta/\log q\) be a point of the open unitary axis at
which \(\mathcal E_s\) and the continued scattering matrices are
holomorphic, and set
\begin{equation}\label{eq:automorphic-renormalized-gram}
 \mathfrak G_N(\theta)
 =\frac1{q+1}W^{-1/2}
   \mathcal E_{s_\theta}^*\chi_N\mathcal E_{s_\theta}W^{-1/2}
   -(2N+1)I_r.
\end{equation}
Then
\begin{equation}\label{eq:automorphic-cesaro-channel-matrix}
 \lim_{L\to\infty}\frac1L\sum_{N=1}^L\mathfrak G_N(\theta)
 =iS(e^{i\theta})^*\partial_\theta S(e^{i\theta})
 =:\mathcal M(\theta)
\end{equation}
in \(M_r(\mathbb C)\).  Moreover,
\begin{equation}\label{eq:automorphic-determinant-trace}
\operatorname{tr}\mathcal M(\theta)
=i\,\partial_\theta\log\det S(e^{i\theta})
=i\,\partial_\theta\log\det\mathcal S_{\mathrm{aut}}(s_\theta),
\end{equation}
where each logarithm is chosen locally.  The traceless part of
\(\mathcal M(\theta)\) is recovered from the
matrix-valued truncated automorphic Gram data and is not, in general,
determined by either determinant.
\end{corollary}

\begin{proof}
Conjugate \eqref{eq:automorphic-maass-selberg} by
\((q+1)^{-1}W^{-1/2}\) on the left and \(W^{-1/2}\) on the right, and use
\[
 S(e^{i\theta})=W^{1/2}\mathcal S_{\mathrm{aut}}(s_\theta)W^{-1/2}.
\]
The two remaining boundary terms are constant matrices times
\(e^{\pm i(2N+1)\theta}\); their Ces\`aro means vanish because
\(0<\theta<\pi\).  This proves
\eqref{eq:automorphic-cesaro-channel-matrix}.  Jacobi's determinant formula
and unitarity give the first equality in
\eqref{eq:automorphic-determinant-trace}; the second follows because the two
scattering matrices are conjugate by the constant matrix \(W^{1/2}\).
The last assertion is Corollary~\ref{cor:maass-selberg-traceless} in
automorphic normalization.
\end{proof}

\begin{remark}[Covariance under a change of cusp origins]\label{rem:height-shift}
The displayed constants use the convention that the first free channel vertex is \(n=1\).  To cut the \(c\)-th cusp \(k_c\ge0\) levels deeper, put \(K_0=\diag(k_1,\ldots,k_r)\) and use the new height \(n'=n-k_c\).  Then
\[
        W'=Wq^{-K_0},
        \qquad
        \mathcal E'_s=\mathcal E_s q^{-K_0s},
\]
where \(q^{-K_0}=\diag(q^{-k_1},\ldots,q^{-k_r})\) and
\(q^{-K_0s}=\diag(q^{-k_1s},\ldots,q^{-k_rs})\).  The classical constant-term
matrix transforms by the explicit rule
\begin{equation}\label{eq:height-shift-automorphic-scattering}
        \mathcal S'_{\mathrm{aut}}(s)
        =q^{K_0(1-s)}\mathcal S_{\mathrm{aut}}(s)q^{-K_0s}.
\end{equation}
On \(\Re s=1/2\), these transformations leave the operator-valued Plancherel
integrand in \eqref{eq:constant-term-plancherel} unchanged.  Thus the height
origin is not a hidden normalization: moving it changes the coefficient matrix
and cusp-width matrix by the displayed compensating factors.
\end{remark}

\begin{proposition}[Identification with classical spherical Eisenstein series]\label{prop:classical-Eisenstein-identification}
Assume Setup~\ref{setup:global-automorphic}, and write
\(E_s^{\mathrm{cl}}=E_s^0=(E_1^0,\ldots,E_r^0)\) for the raw Poincar\'e
family in \eqref{eq:classical-family-defined}, whose incoming
constant-term matrix is \(I_r\) by Lemma
\ref{lem:raw-eisenstein-normalization}.  After passing to
\(\Gamma\backslash G_\infty/\mathcal K\cong V(Y)\), this column family is
precisely \(\mathcal E_s\) in
\eqref{eq:constant-term-Eisenstein-family}.  In particular,
\begin{equation}\label{eq:classical-M-equals-scattering}
        M(s)=\mathcal S_{\mathrm{aut}}(s)
        =W^{-1/2}S(q^{s-1/2})W^{1/2},
\end{equation}
where \(M(s)=(M_{dc}(s))_{d,c}\) is the classical constant-term matrix in
\eqref{eq:classical-component-constant-term}.
\end{proposition}

\begin{proof}
For \(\Re s>1\), each column of
\(E_s^{\mathrm{cl}}\) is right \(\mathcal K\)-invariant.  Lemma
\ref{lem:raw-eisenstein-normalization} shows that it is an eigenfunction of
the unnormalized spherical Hecke operator \(\mathsf T_\infty\) in
\eqref{eq:infinity-hecke-operator}, with incoming coefficient one at
the selected cusp and zero at the others.  Under
\(G_\infty/\mathcal K\cong V(\mathcal T)\), that Hecke operator is the
unnormalized adjacency operator \eqref{eq:quotient-adjacency}, with eigenvalue
\(q^s+q^{1-s}\).  Lemma~\ref{lem:ray-equals-constant-term} identifies its
value on every retained ray with the normalized unipotent constant term, so
the \(c\)-th column has the exact expansion
\[
        q^{ns}e_c+q^{n(1-s)}\mathbf m_c(s)
\]
with \(\mathbf m_c(s)=(M_{dc}(s))_d\).  Applying \(U\) turns the prescribed
incoming term into
\(\sqrt{q+1}\,W^{1/2}\zeta^n e_c\).  Theorem
\ref{thm:exact-scattering} gives the unique generalized eigenfunction with
that incoming amplitude whenever \(F(\zeta)\) is invertible, so its outgoing
coefficient is \(W^{-1/2}S(\zeta)W^{1/2}e_c\).  This is exactly the \(c\)-th
column of \eqref{eq:normalized-constant-term}.  Hence
\(\mathbf m_c(s)=W^{-1/2}S(\zeta)W^{1/2}e_c\) in a nonempty open set.
Meromorphic continuation of the classical Eisenstein series and of the
finite-core solution proves the identity everywhere as a meromorphic
identity, including across removable points of the discrete singular set.
Proposition~\ref{prop:adelic-component-compatibility} shows that the same
identification agrees with the fixed-level adelic Eisenstein family rather
than merely with an abstract solution on one quotient graph.
\end{proof}

\begin{remark}[Full adelic space versus one component]
Proposition~\ref{prop:classical-Eisenstein-identification} is deliberately
componentwise, while Proposition~\ref{prop:adelic-component-compatibility}
records the precise fixed-level adelic interpretation.  Thus the spherical
Eisenstein transform on the full \(K_f\mathcal K\)-fixed adelic space is the
orthogonal direct sum of the component transforms.  If a global Tamagawa
measure is used, its restriction to the \(j\)-th component is a constant
multiple of the local measure used here; multiplying every cusp width on that
component by the same constant restores the global normalization.  The
formulas in this paper use \(\operatorname{vol}(\mathcal K)=1\), for which
\(\lambda(v)=|\Gamma_v|^{-1}\).
\end{remark}

\begin{theorem}[Eisenstein Plancherel formula]\label{thm:automorphic-plancherel}
Let
\[
        D(\theta)=(q-1)^2+4q\sin^2\theta.
\]
With \(E_s^{\mathrm{hol}}:\CC^r\to\mathbb C^{V(Y)}\) denoting the column
family in \eqref{eq:automorphic-asymptotic},
write
\[
        (E_s^{\mathrm{hol}})^*f
        =\bigl(\langle f,E_s^{\mathrm{hol}}(e_c)\rangle\bigr)_{c=1}^r
\]
for finitely supported \(f\), and use the analogous convention for \(\mathcal E_s^*f\).  Then
\begin{equation}\label{eq:automorphic-plancherel}
        P_{\mathrm{ac}}
        =\int_0^\pi E_s^{\mathrm{hol}}\,P_E(\theta)\,(E_s^{\mathrm{hol}})^*\,d\theta,
        \qquad
        P_E(\theta)
        =\frac{q}{2\pi(q+1)D(\theta)}W^{-1},
        \quad
        s=\frac12+\frac{i\theta}{\log q}.
\end{equation}
For the constant-term-normalized family \(\mathcal E_s\) in
\eqref{eq:constant-term-Eisenstein-family}, the same resolution is
\begin{equation}\label{eq:constant-term-plancherel}
        P_{\mathrm{ac}}
        =\int_0^\pi \mathcal E_s\,
          \frac{1}{2\pi(q+1)}W^{-1}\,
          \mathcal E_s^*\,d\theta.
\end{equation}
Both operator integrals are understood weakly on finitely supported vectors
and extend by density.
\end{theorem}

\begin{proof}
The relation between the two generalized-eigenfunction maps is
\[
        UE_s^{\mathrm{hol}}=\Psi_\theta K_{\mathrm{in}}(\zeta).
\]
Insert
\(\Psi_\theta=UE_s^{\mathrm{hol}}K_{\mathrm{in}}(\zeta)^{-1}\)
into \eqref{eq:standard-plancherel}.  On the unit circle,
\[
        |c_{\mathrm{in}}(e^{i\theta})|^2
        =q+q^{-1}-2\cos(2\theta)=\frac{D(\theta)}q.
\]
Therefore
\[
 \frac1{2\pi}
 K_{\mathrm{in}}^{-1}(K_{\mathrm{in}}^{-1})^*
 =\frac{q}{2\pi(q+1)D(\theta)}W^{-1},
\]
which proves \eqref{eq:automorphic-plancherel}.  Since
\(E_s^{\mathrm{hol}}=\mathcal E_s c_{\mathrm{in}}(\zeta)\) and
\(|c_{\mathrm{in}}(e^{i\theta})|^2=D(\theta)/q\), this is equivalent to
\eqref{eq:constant-term-plancherel}.
\end{proof}

\begin{remark}[Contour parameterization]
On the fundamental half-contour
\(C_{\mathrm u}=\{1/2+i\theta/\log q:0\le\theta\le\pi\}\), the two densities
in Theorem~\ref{thm:automorphic-plancherel} become respectively
\[
 \frac{q\log q}{2\pi i(q+1)D(s)}W^{-1}\,ds,
 \qquad
 \frac{\log q}{2\pi i(q+1)}W^{-1}\,ds,
 \quad
 D(s)=(q-q^{2s-1})(q-q^{1-2s}).
\]
This is only the change of variable \(ds=i\,d\theta/\log q\).
\end{remark}

\begin{theorem}[Unitary Eisenstein transform]
\label{thm:unitary-eisenstein-transform}
Let \(L^2_{\mathrm{Eis}}(Y)=\operatorname{ran}P_{\mathrm{ac}}\).  For
finitely supported \(f\), set
\begin{equation}\label{eq:finite-eisenstein-transform}
 (\mathcal F_{\mathrm{Eis},0}f)(\theta)
 =\frac1{\sqrt{q+1}}W^{-1/2}\mathcal E_{s_\theta}^*f,
 \qquad s_\theta=\frac12+\frac{i\theta}{\log q},
\end{equation}
where each pairing in \(\mathcal E_{s_\theta}^*f\) is a finite sum.  This
map depends only on \(P_{\mathrm{ac}}f\) as an almost-everywhere equivalence
class and extends uniquely to a unitary operator
\begin{equation}\label{eq:unitary-eisenstein-transform}
 \mathcal F_{\mathrm{Eis}}:L^2_{\mathrm{Eis}}(Y)
 \longrightarrow
 L^2\!\left((0,\pi),\mathbb C^r;\frac{d\theta}{2\pi}\right).
\end{equation}
It satisfies
\begin{equation}\label{eq:eisenstein-transform-intertwining}
 (\mathcal F_{\mathrm{Eis}}Af)(\theta)
 =(q^{s_\theta}+q^{1-s_\theta})
   (\mathcal F_{\mathrm{Eis}}f)(\theta)
 =2\sqrt q\cos\theta\,
   (\mathcal F_{\mathrm{Eis}}f)(\theta).
\end{equation}
For \(f\in L^2(Y,\lambda)\), write
\(\widehat f_{\mathrm{Eis}}
=\mathcal F_{\mathrm{Eis}}P_{\mathrm{ac}}f\).  Then
\begin{equation}\label{eq:unitary-eisenstein-inverse}
 P_{\mathrm{ac}}f
 =\frac1{2\pi\sqrt{q+1}}
   \int_0^\pi
   \mathcal E_{s_\theta}W^{-1/2}
   \widehat f_{\mathrm{Eis}}(\theta)\,d\theta
\end{equation}
in the weak, equivalently \(L^2\), sense.
\end{theorem}

\begin{proof}
The exact change of normalization is
\[
 U\mathcal E_{s_\theta}
 =\Psi_\theta\sqrt{q+1}\,W^{1/2}.
\]
Hence, for finitely supported \(f\),
\[
 \mathcal F_{\mathrm{Eis},0}f
 =\mathcal F_{\mathrm{ac}}(Uf).
\]
Theorem~\ref{thm:unitary-scattering-transform} now proves unitarity and the
intertwining identity.  Taking the adjoint of
\eqref{eq:unitary-eisenstein-transform} gives
\eqref{eq:unitary-eisenstein-inverse}.
\end{proof}

\begin{corollary}[Fixed-level adelic unitary identification]
\label{cor:adelic-unitary-identification}
For every component in Setup~\ref{setup:global-automorphic}, let
\(Y_j=\Gamma_j\backslash\mathcal T\), let \(W_j\) be its cusp-width matrix,
and denote its constant-term-normalized Eisenstein family by
\(\mathcal E_{j,s}\).  Put
\[
 r_{\mathrm{ad}}=\sum_{j=1}^h r_j,
 \qquad
 W_{\mathrm{ad}}=\bigoplus_{j=1}^h W_j,
 \qquad
 L^2_{\mathrm{Eis,ad}}
 =\bigoplus_{j=1}^h L^2_{\mathrm{Eis}}(Y_j).
\]
Order the columns of \(E_s^{\mathrm{ad}}\) by the bijection
\eqref{eq:adelic-cusp-parametrization}.  Then, initially on vectors whose
component restrictions are finitely supported,
\begin{equation}\label{eq:adelic-unitary-transform}
 (\mathcal F_{\mathrm{Eis}}^{\mathrm{ad}}f)(\theta)
 =\frac1{\sqrt{q+1}}W_{\mathrm{ad}}^{-1/2}
   (E_{s_\theta}^{\mathrm{ad}})^*f
\end{equation}
extends uniquely to a unitary operator
\begin{equation}\label{eq:adelic-unitary-target}
 \mathcal F_{\mathrm{Eis}}^{\mathrm{ad}}:
 L^2_{\mathrm{Eis,ad}}
 \longrightarrow
 L^2\!\left((0,\pi),\mathbb C^{r_{\mathrm{ad}}};
               \frac{d\theta}{2\pi}\right).
\end{equation}
It conjugates \(\mathsf T_\infty\) to multiplication by
\(2\sqrt q\cos\theta\).  Thus the fixed-level spherical automorphic
Eisenstein transform and the orthogonal direct sum of the finite-channel
Jacobi transforms are unitarily identified after precisely the stabilizer
and cusp-width conjugations displayed above.
\end{corollary}

\begin{proof}
Proposition~\ref{prop:adelic-component-compatibility} identifies the
cusp-coordinate adelic family with
\(\bigoplus_jE_{j,s}^{\mathrm{cl}}\), including its measure and Hecke
normalizations.  Proposition~\ref{prop:classical-Eisenstein-identification}
identifies each component family with \(\mathcal E_{j,s}\), and
Theorem~\ref{thm:unitary-eisenstein-transform} gives the corresponding
unitary transform.  Taking their finite orthogonal direct sum proves
\eqref{eq:adelic-unitary-transform}--\eqref{eq:adelic-unitary-target} and the
intertwining assertion.
\end{proof}

\begin{corollary}[Complete spherical automorphic decomposition]
\label{thm:arithmetic-decomposition}
Under Setups~\ref{setup:quotient-adjacency} and
\ref{setup:global-automorphic}, suppose that
\(Y=\Gamma\backslash\mathcal T\) has \(r\) cusps, and put
\(W=\diag(w_1,\ldots,w_r)\).  Then
\begin{equation}\label{eq:automorphic-direct-sum}
 L^2(Y,\lambda)=L^2_{\mathrm{disc}}(Y)\oplus L^2_{\mathrm{Eis}}(Y),
 \qquad
 \Spec_{\mathrm{ess}}(A)=[-2\sqrt q,2\sqrt q],
\end{equation}
where \(L^2_{\mathrm{disc}}(Y)\) denotes the pure point subspace and is
finite-dimensional,
there is no singular continuous spectrum, and the open-band multiplicity is
\(r\) almost everywhere.  If \(\{\varphi_j\}_{j=1}^N\) is an orthonormal
basis of \(L^2_{\mathrm{disc}}(Y)\), then
\begin{equation}
 I
 =\sum_{j=1}^N\varphi_j\varphi_j^*
   +\int_0^\pi \mathcal E_s\,
     \frac{1}{2\pi(q+1)}W^{-1}\,
     \mathcal E_s^*\,d\theta,
 \label{eq:complete-constant-term-resolution}
\end{equation}
where \(s=\frac12+i\theta/\log q\), as a weak identity first on finitely
supported functions and then by density.
\end{corollary}

\begin{proof}
Reduction theory gives a finite graph of groups followed by finitely many
homogeneous cuspidal rays \cite{Serre,Harder,BassLubotzky,Bravo}; hence
Propositions~\ref{prop:unitary-normal-form} and \ref{prop:spectral-type}
apply.  The resolution is Theorem~\ref{thm:automorphic-plancherel} together
with the finite-dimensional point spectrum.
\end{proof}

\begin{remark}[Thresholds]\label{rem:thresholds}
At \(\zeta=\pm1\), the two tail modes coalesce and the boundary-flux factor
\(a(\zeta^{-1}-\zeta)\) vanishes.  All scattering and jump identities above
are statements on the open
band, with endpoint values obtained only when the relevant limits exist.
For \(\sigma\in\{1,-1\}\), a nonzero vector \(u\in\ker F(\sigma)\) gives the
threshold solution with tail
\(g(n)=\sigma^n a^{-1}V^*u\).  This solution is square-integrable exactly when
\(V^*u=0\); it is then supported in the core and contributes an atom at
\(x=2a\sigma\).  If \(V^*u\ne0\), the nonzero constant tail is not in
\(\ell^2\).  More general threshold solutions may contain the second,
coalesced solution \(n\sigma^n\), and their classification requires a
separate endpoint analysis not undertaken here.  Since the endpoints have
zero Lebesgue measure, this qualification does not alter the absolutely
continuous integrals.
\end{remark}

\section{Core-supported point spectrum and determinant cancellation}
\label{sec:cusp-invisible}

The finite determinant also records eigenfunctions that vanish identically on
every cusp.  They belong to the point spectrum but do not occur in the
Eisenstein coefficients.  Separating them is necessary when zeros of the
finite determinant are compared with poles of the scattering matrix.
Exceptional discrete modes that vanish on the noncompact branches, and the
resulting cancellations in scattering or \(L\)-function quotients, already
occur in Chekhov's \(p\)-adic graph scattering
\cite{Chekhov,ChekhovSurvey}.  The statement below packages this phenomenon
as the maximal \(H_C\)-invariant subspace orthogonal to the labelled weighted
cusp attachments and keeps its spectral multiplicities explicit.

\begin{definition}[Core subspace orthogonal to the cusp attachments]
\label{def:cusp-invisible-subspace}
Let
\begin{equation}\label{eq:cusp-invisible-subspace}
        \mathscr H_0
        =\text{the largest \(H_C\)-invariant subspace contained in }\ker V^*.
\end{equation}
Because \(H_C\) is Hermitian, \(\mathscr H_0\) is reducing.  Write
\[
        \CC^d=\mathscr H_0\oplus\mathscr H_0^\perp,
        \qquad
        H_C=H_0\oplus H_{\mathrm{coupled}}.
\]
Since \(\operatorname{ran}V\subset\mathscr H_0^\perp\), let
\(V_{\mathrm{coupled}}:\CC^r\to\mathscr H_0^\perp\) denote \(V\) with its
codomain restricted to \(\mathscr H_0^\perp\).
\end{definition}

\begin{theorem}[Core-supported spectrum and factorization]
\label{thm:core-supported-factorization}
The vectors \((u,0)\) with \(u\in\mathscr H_0\) are exactly the linear span
of all eigenvectors of \(J\) supported entirely in the finite core.  Relative
to the above decomposition,
\begin{equation}\label{eq:F-core-supported-factorization}
        F(\zeta)
        =\bigl(H_0-x(\zeta)I\bigr)\oplus F_{\mathrm{coupled}}(\zeta),
\end{equation}
where
\[
        F_{\mathrm{coupled}}(\zeta)
        =H_{\mathrm{coupled}}-x(\zeta)I
         +\frac{\zeta^{-1}}aV_{\mathrm{coupled}}V_{\mathrm{coupled}}^*.
\]
Hence
\begin{equation}\label{eq:detS-coupled-part}
        \det S(\zeta)
        =(-1)^r
          \frac{\det F_{\mathrm{coupled}}(\zeta^{-1})}
               {\det F_{\mathrm{coupled}}(\zeta)}.
\end{equation}
In particular, every core-supported eigenvalue contributes the
reciprocal-invariant factor \(\det(H_0-x(\zeta)I)\) to \(\det F(\zeta)\), and
that factor cancels identically from the scattering determinant.
\end{theorem}

\begin{proof}
If \(u\in\mathscr H_0\) is an \(H_C\)-eigenvector, then \(V^*u=0\), so
\((u,0)\) is an eigenvector of \(J\).  Conversely, a core-supported
eigenvector satisfies \(H_Cu=\lambda u\) and \(V^*u=0\); the span of all such
vectors is \(H_C\)-invariant and contained in \(\ker V^*\), hence lies in
\(\mathscr H_0\).

Since \(\mathscr H_0\) reduces \(H_C\) and
\(\operatorname{ran}V\subset\mathscr H_0^\perp\), both \(H_C\) and \(VV^*\)
are block diagonal as in \eqref{eq:F-core-supported-factorization}.  The first
factor depends only on \(x(\zeta)=x(\zeta^{-1})\).  Substitution into
\eqref{eq:determinant-identity} cancels this factor and gives
\eqref{eq:detS-coupled-part}.
\end{proof}

\begin{corollary}[Divisor comparison]\label{cor:divisor-comparison}
For \(\eta\in\CC^\times\), let
\[
        m_F(\eta)=\operatorname{ord}_{\eta}\det F,
\]
where the order is zero when \(\det F(\eta)\ne0\).  Then
\begin{equation}\label{eq:divisor-comparison}
        \operatorname{ord}_{\zeta_0}\det S
        =m_F(\zeta_0^{-1})-m_F(\zeta_0).
\end{equation}
Here a positive order denotes a zero and a negative order a pole.  In
particular, if \(\zeta_0\ne\pm1\) is a zero of \(\det F\) of order \(m\) and
\(\det F(\zeta_0^{-1})\ne0\), then \(\det S\) has a pole of order \(m\) at
\(\zeta_0\).  If both reciprocal parameters are zeros, their determinantal
orders cancel to the difference in \eqref{eq:divisor-comparison}.  No
identification of this order with the dimension of a generalized
resonant-state space is asserted here.
\end{corollary}

\begin{proof}
Apply \(\operatorname{ord}_{\zeta_0}\) to
\eqref{eq:determinant-identity}.  The inversion map is biholomorphic on
\(\CC^\times\), so
\[
        \operatorname{ord}_{\zeta_0}\det F(\zeta^{-1})
        =\operatorname{ord}_{\zeta_0^{-1}}\det F.
\]
This gives \eqref{eq:divisor-comparison} and its consequences.
\end{proof}

\begin{remark}[Embedded eigenvalues]\label{rem:embedded-core-spectrum}
If an eigenvalue of \(H_0\) lies in \([-2a,2a]\), it is an embedded eigenvalue
of \(J\).  It is nevertheless orthogonal to every cusp and does not change
\(S(\zeta)\).  This explains algebraically why the finite matrix may be
singular on the unit circle while the scattering matrix remains regular after
the core-supported factor is removed.  The condition agrees with
\cite[Prop.~6.12 and Sec.~7.9]{ArendsPetersonWeich}: an embedded square-integrable
eigenfunction vanishes on all cusps, equivalently at every cusp-attachment
vertex.
\end{remark}

\section{Comparison with a finite resonance matrix}
\label{sec:resonance-comparison}

Arends, Peterson, and Weich introduced a finite matrix whose determinant and
kernel describe resonances and outgoing resonant states on geometrically finite
graphs of groups \cite[Sec.~7.1]{ArendsPetersonWeich}.  Their theory includes
both cuspidal and noncuspidal regular ends.  We use only the cusp case arising
from non-uniform arithmetic lattices in \(\mathrm{PGL}_2\).  The purpose of
this section is not to rederive their resonance set, but to identify the exact
normalization by which their finite matrix yields Eisenstein constant terms,
the unitary transform, and the Maass--Selberg relation proved above.

Let \(\mathcal L\) be their compact core and let \(A_{\mathcal L}\) be its weighted adjacency matrix.  For a cusp-only graph of groups they define
\begin{equation}\label{eq:finite-resonance-matrix}
        H_{\mathrm{res}}(\mu)
        =\frac1{2a}\bigl(A_{\mathcal L}+B(\mu)\bigr)
         -z(\mu)I,
        \qquad
        z(\mu)=\frac{\mu+\mu^{-1}}2,
\end{equation}
where \(B(\mu)\) is the diagonal matrix with
\[
        B(\mu)_{vv}=c_v\frac{a}{\mu}.
\]
Here \(c_v\) is the total cusp attachment multiplicity at \(v\).  Let \(U_C\) be the restriction of the stabilizer-weight conjugation to the core and put
\[
        \widetilde H_{\mathrm{res}}(\mu)
        =U_CH_{\mathrm{res}}(\mu)U_C^{-1}.
\]
Label the \(r\) geometric cusp channels.  If \(\varepsilon_c\) is the first
directed cusp edge, attached to the core at \(v(c)\), use the stabilizers from
Setup~\ref{setup:quotient-adjacency} and set
\[
        m_c=\frac{|\Gamma_{v(c)}|}{|\Gamma_{\varepsilon_c}|}.
\]
This is the directed cusp-attachment index, so
\(c_v=\sum_{c:v(c)=v}m_c\).  If \(e_c\) and \(e_v\) are the standard coordinate vectors of \(\CC^r\) and \(\CC^d\), respectively, define
\begin{equation}\label{eq:attachment-C}
        C e_c=\sqrt{m_c}\,e_{v(c)},
        \qquad
        CC^*=\diag(c_v),
        \qquad
        V=aC.
\end{equation}

\begin{theorem}[Exact normalization and recovery of Eisenstein coefficients]
\label{thm:finite-resonance-comparison}
With the common parameter
\[
        \zeta=\mu=q^{s-1/2},
\]
the finite matrices satisfy the exact identity
\begin{equation}\label{eq:F-equals-resonance-matrix}
        F(\mu)=2a\,\widetilde H_{\mathrm{res}}(\mu).
\end{equation}
Consequently, whenever \(\det H_{\mathrm{res}}(\mu)\ne0\), the labelled
Jacobi scattering matrix is
\begin{equation}\label{eq:S-from-resonance-matrix}
        S(\mu)
        =-I_r+\frac{\mu^{-1}-\mu}{2}
          C^*\widetilde H_{\mathrm{res}}(\mu)^{-1}C,
\end{equation}
and its determinant is
\begin{equation}\label{eq:detS-from-resonance-matrix}
        \det S(\mu)
        =(-1)^r
          \frac{\det H_{\mathrm{res}}(\mu^{-1})}
               {\det H_{\mathrm{res}}(\mu)}.
\end{equation}
The classical constant-term matrix is therefore
\begin{equation}\label{eq:automorphic-scattering-from-resonance-matrix}
        \mathcal S_{\mathrm{aut}}(s)
        =W^{-1/2}
          \left[-I_r+\frac{\mu^{-1}-\mu}{2}
          C^*\widetilde H_{\mathrm{res}}(\mu)^{-1}C\right]
          W^{1/2}.
\end{equation}
Thus the finite resonance matrix, together with the labelled attachment
operator \(C\), determines all Eisenstein constant-term coefficients.  Its
determinant alone determines only the scattering determinant through
\eqref{eq:detS-from-resonance-matrix}.
If
\[
        m_{\mathrm{res}}(\eta)
        =\operatorname{ord}_{\mu=\eta}\det H_{\mathrm{res}}(\mu),
\]
then its determinantal resonance multiplicities determine the scattering-determinant divisor by
\begin{equation}\label{eq:resonance-divisor-comparison}
        \operatorname{ord}_{\mu_0}\det S
        =m_{\mathrm{res}}(\mu_0^{-1})
         -m_{\mathrm{res}}(\mu_0).
\end{equation}
All displayed identities extend meromorphically in \(\mu\in\CC^\times\), with the thresholds \(\mu=\pm1\) interpreted by limits when they exist.
\end{theorem}

\begin{proof}
Conjugation by \(U_C\) turns \(A_{\mathcal L}\) into \(H_C\), while the diagonal matrix \(B(\mu)\) is unchanged.  By \eqref{eq:attachment-C},
\[
\begin{aligned}
 2a\,\widetilde H_{\mathrm{res}}(\mu)
 &=H_C-a(\mu+\mu^{-1})I+a\mu^{-1}CC^*\\
 &=H_C-x(\mu)I+\frac{\mu^{-1}}aVV^*
 =F(\mu).
\end{aligned}
\]
Substitution into \eqref{eq:S-from-F} gives
\eqref{eq:S-from-resonance-matrix}.  Formula
\eqref{eq:detS-from-resonance-matrix} follows from Theorem
\ref{thm:determinant-identity}, because similarity and the scalar factor
\(2a\) cancel in the determinant quotient.  Equation
\eqref{eq:automorphic-scattering-from-resonance-matrix} is
\eqref{eq:automorphic-scattering-dictionary} applied to
\eqref{eq:S-from-resonance-matrix}, and
\eqref{eq:resonance-divisor-comparison} is Corollary
\ref{cor:divisor-comparison} under
\eqref{eq:F-equals-resonance-matrix}.
\end{proof}

\begin{remark}[Prior input and present output]\label{rem:prior-input}
The meromorphic continuation of the resolvent, the characterization of
outgoing resonant states, the finite resonance matrix, and the explicit
resonance computations belong to \cite{ArendsPetersonWeich}; finite-tail
scattering and reciprocal determinant formulas are also classical
\cite{VarbanovBrun,Chekhov}.  Here these inputs are converted, with exact
normalization, into the labelled automorphic coefficient matrix, the unitary
multi-cusp transform, the matrix Maass--Selberg term, and the separation of
core-supported factors.  No independent priority is claimed for the
reciprocal quotient itself.
\end{remark}

\begin{remark}[Automorphic interpretation]\label{rem:automorphic-interpretation}
The \(r\) columns of \(\mathcal E_s\) are normalized by their incoming
constant terms, while \(E_s^{\mathrm{hol}}\) uses Efrat's holomorphic local
factors.  Equations \eqref{eq:automorphic-scattering-dictionary} and
\eqref{eq:Phi-S-dictionary} are the exact changes of basis from the standard
Jacobi waves to these two automorphic conventions.  In the quotient-graph
Hilbert space, eigenfunctions supported in the finite core vanish on every
cusp.  Their spherical constant terms satisfy the second-order adjacency
recurrence in the height variable.  Since those constant terms vanish for all
sufficiently large heights, the recurrence forces them to vanish identically;
hence, in the spherical automorphic realization, these eigenfunctions belong
to the cuspidal subspace, namely the subspace with zero constant term at every
cusp.  Residual \(L^2\)-eigenfunctions, obtained from residues at poles, may
instead arise from the meromorphically continued Eisenstein family.
\end{remark}

\section{Arithmetic applications}\label{sec:examples}

\subsection{The Nagao quotient}\label{subsec:nagao-calibration}

For \(\Gamma=\mathrm{PGL}_2(\Fq[t])\), the quotient is a ray.  It is useful
to distinguish the convenient ray normalization
\[
 \lambda_{\mathrm{ray}}(0)=1,
 \qquad
 \lambda_{\mathrm{ray}}(n)=(q+1)q^{-n}\quad(n\ge1)
\]
from the Haar normalization fixed in Setup~\ref{setup:quotient-adjacency}.
Since the initial projective stabilizer has order \(q(q^2-1)\),
\begin{equation}\label{eq:nagao-haar-ray-scaling}
 \lambda_{\mathrm{Haar}}
 =\frac1{q(q^2-1)}\lambda_{\mathrm{ray}},
 \qquad
 W_{\mathrm{ray}}=(1),
 \qquad
 W_{\mathrm{Haar}}=\left(\frac1{q(q^2-1)}\right).
\end{equation}
In the standard Jacobi model one has
\[
        d=r=1,\qquad H_C=0,\qquad V=a\sqrt{q+1}.
\]
Then
\begin{equation}\label{eq:Nagao-F}
        F(\zeta)=a(q\zeta^{-1}-\zeta),
\end{equation}
and Theorem \ref{thm:exact-scattering} gives
\begin{equation}\label{eq:Nagao-S}
        S(\zeta)
        =-\frac{q\zeta-\zeta^{-1}}{q\zeta^{-1}-\zeta}.
\end{equation}
Using \eqref{eq:Phi-S-dictionary} gives
\[
        \Phi(s)=1.
\]
Thus the nontrivial reflection coefficient in the standard Jacobi basis,
\eqref{eq:Nagao-S}, becomes the scalar coefficient \(1\) in Efrat's
holomorphic ray basis.  In the constant-term-normalized basis, however,
\[
        \mathcal S_{\mathrm{aut}}(s)
        =S(\zeta)
        =\frac{c_{\mathrm{out}}(\zeta)}{c_{\mathrm{in}}(\zeta)}.
\]
This distinction is exactly the local normalization factor separated in Theorem \ref{thm:normalization-dictionary}.

The zeros \(\zeta=\pm\sqrt q\) of \eqref{eq:Nagao-F} correspond to the two
simple eigenvalues \(\pm(q+1)\).  Put
\[
 D(\theta)=(q-1)^2+4q\sin^2\theta.
\]
Theorem~\ref{thm:automorphic-plancherel} gives, in the ray normalization, the
Efrat-holomorphic density
\[
        \frac{q}{2\pi(q+1)D(\theta)}\,d\theta,
\]
whereas the constant-term-normalized density is
\([2\pi(q+1)]^{-1}d\theta\).  In the actual quotient Haar measure, the two
densities are respectively
\begin{equation}\label{eq:nagao-two-haar-densities}
 \frac{q^2(q-1)}{2\pi D(\theta)}\,d\theta,
 \qquad
 \frac{q(q-1)}{2\pi}\,d\theta,
\end{equation}
respectively.  Thus the ray and Haar versions differ only by the global
Hilbert-space scale \(q(q^2-1)\), while the spectral parameter and the two
coefficient normalizations are unchanged.

We finally record the convention-sensitive prefactor in
\cite[Thm.~5.3, Eq.~(7)]{EfratAutomorphicSpectra}.  Let \(h_\theta\) be the
ray generalized eigenfunction normalized by
\(h_\theta(0)=i(q+1)\sin\theta\).  In Efrat's measure
\(\mu_{\mathrm E}(0)=1/(q+1)\), the \((0,0)\)-entry of the identity kernel is
\(q+1\).  The two exceptional eigenprojections contribute
\((q^2-1)/q\), and
\[
 \int_0^\pi\frac{\sin^2\theta}{D(\theta)}\,d\theta
 =\frac{\pi}{2q(q+1)},
 \qquad
 |h_\theta(0)|^2=(q+1)^2\sin^2\theta.
\]
The missing continuous contribution is therefore obtained with coefficient
\(2/\pi\), not \(2\pi\); the factor \(D(\theta)\) remains in the denominator.
This single kernel entry suffices to fix the prefactor, while the full
decomposition follows from the general unitary transform above.

\subsection{The two-cusp Hecke quotient \(\Gamma_0(T)\)}\label{subsec:gamma0T}

Let \(A_0=\Fq[T]\), let
\(K_\infty=\mathbb F_q(\!(T^{-1})\!)\), and let
\[
 \widetilde\Gamma_0(T)=
 \left\{
 \begin{pmatrix}\alpha&\beta\\\gamma&\delta\end{pmatrix}
 \in\mathrm{GL}_2(A_0):\gamma\in TA_0\right\},
 \qquad
 \Gamma_0(T)=\widetilde\Gamma_0(T)/\Fq^\times
 \subset\mathrm{PGL}_2(K_\infty).
\]
Passing to the projective quotient does not change the tree action, since scalar matrices act trivially.
Bravo proves that the underlying quotient graph
\(\Gamma_0(T)\backslash\mathcal T\) is a double ray
\cite[Prop.~8.5]{Bravo}; its two ends are the cusps represented by \(0\) and
\(\infty\).  The following direct calculation supplies the stabilizer data
that the unweighted double ray does not record.  Put
\(\mathcal O_\infty=\Fq[\![T^{-1}]\!]\),
\(\mathcal K_\infty=\mathrm{PGL}_2(\mathcal O_\infty)\), and
\[
 \omega_T=\begin{pmatrix}0&-1\\T&0\end{pmatrix}.
\]

\begin{lemma}[Central and first-cusp stabilizers for \(\Gamma_0(T)\)]
\label{lem:gamma0-stabilizers}
Let \(v_0=\mathcal K_\infty\) and
\(v_1=\omega_T\mathcal K_\infty\).  These are adjacent lifts of the two
central quotient vertices.  Their projective stabilizers are opposite Borel
subgroups \(B_+,B_-<\mathrm{PGL}_2(\Fq)\), and the stabilizer of the edge
between them is their common diagonal torus \(\mathbb T\).  Consequently,
both directed indices across the central quotient edge are
\[
 [B_\pm:\mathbb T]=q.
\]
The other edge orbit at either central vertex is the outward cusp edge and
has directed index one from the central vertex.  The stabilizer of the first
vertex beyond that edge has order \(q^2(q-1)\) in \(\Gamma_0(T)\); the reverse
directed index is \(q\).
\end{lemma}

\begin{proof}
An element of \(\widetilde\Gamma_0(T)\) has the form
\[
 \gamma=\begin{pmatrix}\alpha&\beta\\Tc&\delta\end{pmatrix},
 \qquad \alpha,\beta,c,\delta\in\Fq[T],
 \qquad \det\gamma\in\Fq^\times.
\]
Membership in the stabilizer of \(v_0\) forces every entry to be integral at
infinity.  Hence \(\alpha,\beta,\delta\) are constant and \(c=0\), which gives
\(B_+\) after quotienting by scalar matrices.  Directly,
\[
 \omega_T^{-1}\gamma\omega_T
 =\begin{pmatrix}\delta&-c\\-T\beta&\alpha\end{pmatrix},
\]
so \(\omega_T\) normalizes \(\Gamma_0(T)\) projectively and sends \(v_0\) to
\(v_1\), giving the opposite Borel \(B_-\).  The intersection stabilizing the
common edge is the diagonal torus, of order \(q-1\); since
\(|B_\pm|=q(q-1)\), the central indices are \(q\).

For the outward vertex \(v_\infty=\operatorname{diag}(T,1)\mathcal K_\infty\),
the condition
\[
 \operatorname{diag}(T^{-1},1)\gamma
 \operatorname{diag}(T,1)
 =\begin{pmatrix}\alpha&\beta/T\\T^2c&\delta\end{pmatrix}
 \in\mathrm{PGL}_2(\mathcal O_\infty)
\]
is equivalent to \(c=0\), \(\alpha,\delta\in\Fq^\times\), and
\(\deg\beta\le1\).  There are therefore \(q^2(q-1)\) projective stabilizer
elements.  The intervening edge stabilizer is \(B_+\), of order \(q(q-1)\),
which gives the two stated directed indices.  Applying \(\omega_T\) gives the
same computation at the other cusp.
\end{proof}

Figure~\ref{fig:gamma0-double-ray} displays both the directed quotient data
and the symmetric Jacobi couplings obtained after stabilizer conjugation.

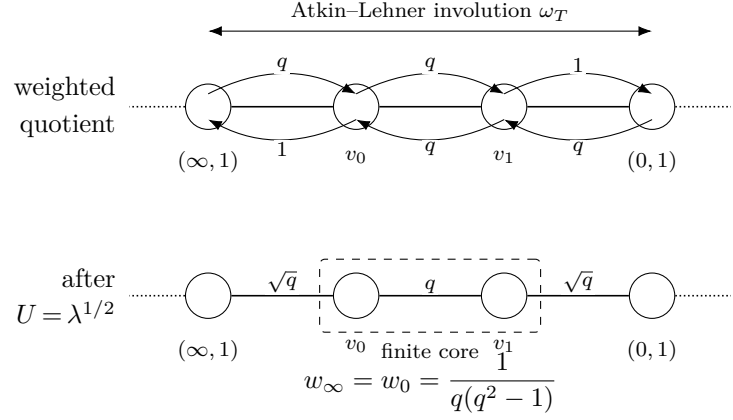
\begin{figure}[htbp]
\centering
\begin{tikzpicture}[x=1.45cm,y=1cm,>=Latex,
 vertex/.style={circle,draw,fill=white,minimum size=6mm,inner sep=0pt},
 edge/.style={semithick},
 dottededge/.style={densely dotted,semithick},
 dir/.style={-{Latex[length=1.8mm]},thin},
 lab/.style={font=\scriptsize,fill=white,inner sep=.8pt}
]
\node[font=\small,anchor=east,align=right] at (-.75,.65)
 {weighted\\quotient};
\node[vertex] (cinf) at (0,.65) {};
\node[vertex] (v0) at (1.35,.65) {};
\node[vertex] (v1) at (2.70,.65) {};
\node[vertex] (c0) at (4.05,.65) {};
\node[below=4pt of cinf,font=\scriptsize] {\((\infty,1)\)};
\node[below=4pt of v0,font=\scriptsize] {\(v_0\)};
\node[below=4pt of v1,font=\scriptsize] {\(v_1\)};
\node[below=4pt of c0,font=\scriptsize] {\((0,1)\)};
\draw[dottededge] (cinf) -- +(-.72,0);
\draw[edge] (cinf)--(v0)--(v1)--(c0);
\draw[dottededge] (c0) -- +(.72,0);

\draw[dir] ($(cinf)+(0,.17)$) to[bend left=27]
 node[lab,above] {\(q\)} ($(v0)+(0,.17)$);
\draw[dir] ($(v0)+(0,-.17)$) to[bend left=27]
 node[lab,below] {\(1\)} ($(cinf)+(0,-.17)$);
\draw[dir] ($(v0)+(0,.17)$) to[bend left=27]
 node[lab,above] {\(q\)} ($(v1)+(0,.17)$);
\draw[dir] ($(v1)+(0,-.17)$) to[bend left=27]
 node[lab,below] {\(q\)} ($(v0)+(0,-.17)$);
\draw[dir] ($(v1)+(0,.17)$) to[bend left=27]
 node[lab,above] {\(1\)} ($(c0)+(0,.17)$);
\draw[dir] ($(c0)+(0,-.17)$) to[bend left=27]
 node[lab,below] {\(q\)} ($(v1)+(0,-.17)$);

\draw[{Latex[length=1.8mm]}-{Latex[length=1.8mm]},thin]
 ($(cinf)+(0,1.0)$) -- node[above,font=\scriptsize]
 {Atkin--Lehner involution \(\omega_T\)} ($(c0)+(0,1.0)$);

\node[font=\small,anchor=east,align=right] at (-.75,-1.85)
 {after\\\(U{\,=\,}\lambda^{1/2}\)};
\node[vertex] (jcinf) at (0,-1.85) {};
\node[vertex] (jv0) at (1.35,-1.85) {};
\node[vertex] (jv1) at (2.70,-1.85) {};
\node[vertex] (jc0) at (4.05,-1.85) {};
\node[below=4pt of jcinf,font=\scriptsize] {\((\infty,1)\)};
\node[below=4pt of jv0,font=\scriptsize] {\(v_0\)};
\node[below=4pt of jv1,font=\scriptsize] {\(v_1\)};
\node[below=4pt of jc0,font=\scriptsize] {\((0,1)\)};
\draw[dottededge] (jcinf) -- +(-.72,0);
\draw[edge] (jcinf) -- node[lab,above] {\(\sqrt q\)} (jv0);
\draw[edge] (jv0) -- node[lab,above] {\(q\)} (jv1);
\draw[edge] (jv1) -- node[lab,above] {\(\sqrt q\)} (jc0);
\draw[dottededge] (jc0) -- +(.72,0);
\node[draw,dashed,rounded corners=2pt,fit=(jv0)(jv1),inner sep=5pt,
 label={[font=\scriptsize]below:finite core}] {};

\node[font=\small] at (2.03,-3.0)
 {\(w_\infty=w_0=\dfrac{1}{q(q^2-1)}\)};
\end{tikzpicture}
\caption{The two-cusp quotient for \(\Gamma_0(T)\).  The upper row shows
the two directed indices on each edge adjacent to the central vertices.  The
lower row shows the standard-\(\ell^2\) Jacobi realization: the central
coupling is \(q\), both cusp attachments are \(\sqrt q\), and the two cusp
widths agree.}
\label{fig:gamma0-double-ray}
\end{figure}

The central directed indices and the fixed outward edges give local degree
\(q+1\) at both central vertices.  Absorb \(v_0,v_1\) into the finite core and
start each free channel at the next vertex.  In the standard \(\ell^2\)
normalization, Lemma~\ref{lem:gamma0-stabilizers} gives the following core and
labelled attachment data:
\begin{equation}\label{eq:gamma0-core}
 H_C=\begin{pmatrix}0&q\\q&0\end{pmatrix},
 \qquad C_{\mathrm{att}}=I_2,
 \qquad V=\sqrt q\,I_2.
\end{equation}
For the stabilizer-volume measure \(\lambda(v)=|\Gamma_{0}(T)_v|^{-1}\), the two free tails have the same width,
\begin{equation}\label{eq:gamma0-width}
        W=wI_2,
        \qquad w=\frac{1}{q(q^2-1)}.
\end{equation}
Indeed, Lemma~\ref{lem:gamma0-stabilizers} gives projective stabilizer order
\(q^2(q-1)\) at the first free vertex on either side, and
\(w(q+1)q^{-1}=1/[q^2(q-1)]\).

\begin{proposition}[Exact \(\Gamma_0(T)\) scattering matrix]\label{prop:gamma0-scattering}
Put \(a=\sqrt q\) and use \(\mu=\zeta=q^{s-1/2}\).  Then
\begin{align}
 F_{\Gamma_0(T)}(\mu)
 &=\begin{pmatrix}-a\mu&q\\q&-a\mu\end{pmatrix},
 &\det F_{\Gamma_0(T)}(\mu)&=q(\mu^2-q),
 \label{eq:gamma0-F}\\
 S_{\Gamma_0(T)}(\mu)
 &=\frac1{\mu^2-q}
 \begin{pmatrix}
 q-1&\displaystyle\frac{a(\mu^2-1)}{\mu}\\[4pt]
 \displaystyle\frac{a(\mu^2-1)}{\mu}&q-1
 \end{pmatrix},
 \label{eq:gamma0-S}\\
 \det S_{\Gamma_0(T)}(\mu)
 &=\frac{1-q\mu^2}{\mu^2(\mu^2-q)}.
 \label{eq:gamma0-detS}
\end{align}
Since \(W\) is scalar, this is simultaneously the Jacobi scattering matrix and
the classical automorphic constant-term matrix.
\end{proposition}

\begin{proof}
The term \(\mu^{-1}VV^*/a\) in \(F\) cancels the
\(-a\mu^{-1}I_2\) part of \(-x(\mu)I_2\), so
\eqref{eq:gamma0-core} gives \eqref{eq:gamma0-F}.  Inverting this
\(2\times2\) matrix in \eqref{eq:S-from-F} gives \eqref{eq:gamma0-S}; its
determinant is \eqref{eq:gamma0-detS}, in agreement with the reciprocal
quotient \eqref{eq:determinant-identity}.
\end{proof}

The Atkin--Lehner element \(\omega_T\) interchanges the two cusps.
Accordingly, the symmetric and antisymmetric channel vectors diagonalize
\eqref{eq:gamma0-S}, with eigenvalues
\begin{equation}\label{eq:gamma0-AL-channels}
 \sigma_+(\mu)=\frac{a\mu-1}{\mu(\mu-a)},
 \qquad
 \sigma_-(\mu)=-\frac{a\mu+1}{\mu(\mu+a)}.
\end{equation}
The poles \(\mu=\pm\sqrt q\) give the two exceptional \(L^2\)-eigenvalues \(\pm(q+1)\), while the reciprocal points \(\mu=\pm q^{-1/2}\) are zeros of the scattering determinant.  On \(|\mu|=1\), the reflection and transmission amplitudes are
\[
 \mathcal R(\mu)=\frac{q-1}{\mu^2-q},
 \qquad
 \tau(\mu)=\frac{\sqrt q(\mu^2-1)}{\mu(\mu^2-q)},
\]
and Corollary~\ref{cor:S-functional-unitary} gives the two-channel unitarity
relations.  The nonzero transmission coefficient is forced by the central
Borel--torus indices, so this arithmetic quotient realizes genuine
two-channel scattering with absolutely continuous multiplicity two.

\begin{corollary}[Channel-resolved Maass--Selberg term for \(\Gamma_0(T)\)]
\label{cor:gamma0-channel-mass}
Let
\[
 e_+=\frac{(1,1)^t}{\sqrt2},\qquad
 e_-=\frac{(1,-1)^t}{\sqrt2},
\]
and put
\(\mathcal M_{\Gamma_0}(\theta)
=iS_{\Gamma_0(T)}(e^{i\theta})^*
\partial_\theta S_{\Gamma_0(T)}(e^{i\theta})\).
Then
\begin{equation}\label{eq:gamma0-channel-mass-eigenvalues}
 \mathcal M_{\Gamma_0}(\theta)e_\pm=m_\pm(\theta)e_\pm,
\end{equation}
where
\begin{equation}\label{eq:gamma0-channel-mass-scalars}
 m_+(\theta)=\frac{2(1-\sqrt q\cos\theta)}
                    {q+1-2\sqrt q\cos\theta},
 \qquad
 m_-(\theta)=\frac{2(1+\sqrt q\cos\theta)}
                    {q+1+2\sqrt q\cos\theta}.
\end{equation}
In particular,
\begin{equation}\label{eq:gamma0-traceless-mass}
 m_+(\theta)-m_-(\theta)
 =-\frac{4\sqrt q\,(q-1)\cos\theta}
          {(q+1)^2-4q\cos^2\theta}.
\end{equation}
The logarithmic derivative of \(\det S_{\Gamma_0(T)}\) supplies
\(m_++m_-\), whereas the scalar trace identity alone does not supply the
channel difference \eqref{eq:gamma0-traceless-mass}.
\end{corollary}

\begin{proof}
The constant Atkin--Lehner basis \((e_+,e_-)\) diagonalizes the scattering
matrix with eigenvalues \(\sigma_+,\sigma_-\) in
\eqref{eq:gamma0-AL-channels}.  On \(\mu=e^{i\theta}\),
\[
 i\overline{\sigma_\pm(\mu)}\,
 \partial_\theta\sigma_\pm(\mu)
 =-\mu\,\partial_\mu\log\sigma_\pm(\mu).
\]
Substitution of \eqref{eq:gamma0-AL-channels} gives
\eqref{eq:gamma0-channel-mass-scalars}, and subtraction gives
\eqref{eq:gamma0-traceless-mass}.  The last assertion follows from
\eqref{eq:trace-mass-determinant}.
\end{proof}

\subsection{A four-cusp elliptic quotient over \(\mathbb F_3\)}\label{subsec:F3-elliptic}

Let
\begin{equation}\label{eq:F3-curve}
        X:\ y^2=x^3+x+1
\end{equation}
be the elliptic curve over \(\mathbb F_3\), with point at infinity \(O\), and set
\[
 \mathcal A_X=H^0(X\setminus\{O\},\mathcal O_X)
   \cong\mathbb F_3[x,y]/(y^2-x^3-x-1),
 \qquad
 \Gamma=\mathrm{PGL}_2(\mathcal A_X).
\]
Here \(\mathrm{PGL}_2(\mathcal A_X)\) denotes the image of
\(\mathrm{GL}_2(\mathcal A_X)\) in
\(\mathrm{PGL}_2(\mathbb F_3(X))\), subsequently
embedded in the group over the completion at \(O\).
The notation \(\operatorname{Pic}(X)\) denotes the divisor class group of
\(X\) over \(\mathbb F_3\), and \(\operatorname{Pic}(\mathcal A_X)\) denotes
the ideal class group of the Dedekind domain \(\mathcal A_X\), equivalently
the divisor class group of the affine curve \(X\setminus\{O\}\).
The rational points are
\[
        X(\mathbb F_3)=\{O,(0,1),(0,2),(1,0)\}.
\]
Thus \(\#X(\mathbb F_3)=4\).  Since \((1,0)\) is the only nonzero rational two-torsion point, \(X(\mathbb F_3)\cong\mathbb Z/4\mathbb Z\).  Because \(O\) is rational, the degree splitting gives
\[
        \operatorname{Pic}(\mathcal A_X)
        \cong\operatorname{Pic}(X)/\langle[O]\rangle
        \cong\operatorname{Pic}^0(X)(\mathbb F_3)
        \cong X(\mathbb F_3)\cong\mathbb Z/4\mathbb Z.
\]
The cusps of
\(\Gamma\backslash\mathcal T\) are indexed by
\(\operatorname{Pic}(\mathcal A_X)\)
\cite[Sec.~3.4]{ArendsPetersonWeich}, so the quotient has exactly four cusps.
We order them geometrically as west, south, northeast, and southeast in the
stabilizer-labelled Takahashi domain reproduced in
\cite[Fig.~7.3]{ArendsPetersonWeich}; this order is denoted
\((c_1,c_2,c_3,c_4)\).

\subsubsection{The weighted core and projective stabilizer normalization}

Choose core vertices \(v_1,\ldots,v_9\) so that \(v_1,v_8,v_6\) are
respectively the west, south, and common eastern attachment vertices.
Conjugation by \(\lambda(v)^{1/2}=|\Gamma_v|^{-1/2}\) changes a core edge
\(e=vw\) into the symmetric Jacobi coupling
\begin{equation}\label{eq:stabilizer-to-jacobi}
        h_{vw}=\frac{\sqrt{|\Gamma_v||\Gamma_w|}}{|\Gamma_e|}.
\end{equation}
The labels in Takahashi's fundamental domain \cite{Takahashi}, reproduced in
\cite[Fig.~7.3]{ArendsPetersonWeich}, are orders of stabilizers for the
\(\mathrm{GL}_2(\mathcal A_X)\)-action.  Passing to
\(\Gamma=\mathrm{PGL}_2(\mathcal A_X)\) quotients every vertex and edge
stabilizer by the central unit group
\(\mathcal A_X^\times=\mathbb F_3^\times\), of order two; compare the
linear and projective stabilizers in \cite[Sec.~1]{Knudson}.  Thus the
projective stabilizer orders in our vertex numbering are
\begin{equation}\label{eq:F3-core-vertex-orders}
\begingroup
\setlength{\arraycolsep}{4.2pt}
\begin{array}{c|ccccccccc}
v&v_1&v_2&v_3&v_4&v_5&v_6&v_7&v_8&v_9\\ \hline
|\Gamma_v|&6&24&3&1&4&2&3&6&24
\end{array}
\endgroup
\end{equation}
The core-edge stabilizer orders are
\begin{equation}\label{eq:F3-core-edge-orders}
\begingroup
\setlength{\arraycolsep}{3.8pt}
\begin{array}{c|cccccccc}
e&v_1v_2&v_1v_3&v_3v_4&v_4v_5&v_4v_6&v_4v_7&v_7v_8&v_8v_9\\ \hline
|\Gamma_e|&6&3&1&1&1&1&3&6
\end{array}
\endgroup
\end{equation}
For the ordered cusps \((c_1,c_2,c_3,c_4)\), denote by \(e_{c_j}\) the edge
from the core to the first free vertex \((c_j,1)\).  The attachment-edge
orders and the stabilizer orders at the first free vertices are, respectively,
\begin{equation}\label{eq:F3-attachment-stabilizer-orders}
 \bigl(|\Gamma_{e_{c_1}}|,\ldots,|\Gamma_{e_{c_4}}|\bigr)
 =(6,6,2,2),
 \qquad
 \bigl(|\Gamma_{(c_1,1)}|,\ldots,|\Gamma_{(c_4,1)}|\bigr)
 =(18,18,6,6).
\end{equation}
The division by the common central factor leaves every stabilizer index, and
hence the weighted adjacency and the resonance matrix, unchanged.  It is
nevertheless essential for the absolute Haar volumes used in the Plancherel
formula.
For example,
\[
 h_{v_1v_2}=\frac{\sqrt{6\cdot24}}6=2,
 \qquad
 h_{v_1v_3}=\frac{\sqrt{6\cdot3}}3=\sqrt2,
 \qquad
 h_{v_3v_4}=\frac{\sqrt{3\cdot1}}1=\sqrt3.
\]
Applying \eqref{eq:stabilizer-to-jacobi} to every entry of
\eqref{eq:F3-core-edge-orders} gives
\begin{equation}\label{eq:F3-Hcore}
\begingroup
\setlength{\arraycolsep}{3.2pt}
H_C=
\begin{pmatrix}
0&2&\sqrt2&0&0&0&0&0&0\\
2&0&0&0&0&0&0&0&0\\
\sqrt2&0&0&\sqrt3&0&0&0&0&0\\
0&0&\sqrt3&0&2&\sqrt2&\sqrt3&0&0\\
0&0&0&2&0&0&0&0&0\\
0&0&0&\sqrt2&0&0&0&0&0\\
0&0&0&\sqrt3&0&0&0&\sqrt2&0\\
0&0&0&0&0&0&\sqrt2&0&2\\
0&0&0&0&0&0&0&2&0
\end{pmatrix}.
\endgroup
\end{equation}
The four labelled attachments are
\begin{equation}\label{eq:F3-attachment}
        C_{\mathrm{att}}
        =\begin{pmatrix}e_1&e_8&e_6&e_6\end{pmatrix},
        \qquad
        V=\sqrt3\,C_{\mathrm{att}}.
\end{equation}
Here \(e_j\) is the \(j\)-th standard coordinate vector of \(\CC^9\).
Indeed, \eqref{eq:F3-attachment-stabilizer-orders} gives the four attachment
couplings
\[
 \frac{\sqrt{6\cdot18}}6,\quad
 \frac{\sqrt{6\cdot18}}6,\quad
 \frac{\sqrt{2\cdot6}}2,\quad
 \frac{\sqrt{2\cdot6}}2,
\]
all equal to \(\sqrt3\).  Thus the two eastern cusps attach to the same core
vertex but remain distinct labelled channels.  Figure
\ref{fig:F3-standard-core} records the sparse Jacobi data derived from
\eqref{eq:F3-core-vertex-orders}--
\eqref{eq:F3-attachment-stabilizer-orders}.

\begin{figure}[htbp]
\centering
\begin{tikzpicture}[x=1.05cm,y=1.0cm,>=Latex,
 core/.style={circle,draw,fill=white,inner sep=1.8pt,font=\scriptsize},
 lab/.style={font=\scriptsize,fill=white,inner sep=1pt},
 channel/.style={dashed,->,thick}]
\node[core] (v1) at (0,0) {$v_1$};
\node[core] (v2) at (0,1.25) {$v_2$};
\node[core] (v3) at (1.55,0) {$v_3$};
\node[core] (v4) at (3.1,0) {$v_4$};
\node[core] (v5) at (3.1,1.25) {$v_5$};
\node[core] (v6) at (4.65,0) {$v_6$};
\node[core] (v7) at (3.1,-1.25) {$v_7$};
\node[core] (v8) at (3.1,-2.5) {$v_8$};
\node[core] (v9) at (1.55,-2.5) {$v_9$};
\draw (v1)--node[lab,left]{$2$}(v2);
\draw (v1)--node[lab,above]{$\sqrt2$}(v3);
\draw (v3)--node[lab,above]{$\sqrt3$}(v4);
\draw (v4)--node[lab,left]{$2$}(v5);
\draw (v4)--node[lab,above]{$\sqrt2$}(v6);
\draw (v4)--node[lab,left]{$\sqrt3$}(v7);
\draw (v7)--node[lab,left]{$\sqrt2$}(v8);
\draw (v8)--node[lab,above]{$2$}(v9);
\draw[channel] (v1)--node[lab,above]{$\sqrt3$}++(-1.35,0) node[left,lab]{$c_1$};
\draw[channel] (v8)--node[lab,right]{$\sqrt3$}++(0,-1.15) node[below,lab]{$c_2$};
\draw[channel] (v6)--node[lab,above left]{$\sqrt3$}++(1.15,0.85) node[right,lab]{$c_3$};
\draw[channel] (v6)--node[lab,below left]{$\sqrt3$}++(1.15,-0.85) node[right,lab]{$c_4$};
\end{tikzpicture}
\caption{The standard-\(\ell^2\) finite core for \eqref{eq:F3-curve}, redrawn
from the stabilizer-labelled Takahashi domain in
\cite[Fig.~7.3]{ArendsPetersonWeich}.  Solid-edge labels are the conjugated
couplings \eqref{eq:stabilizer-to-jacobi}; dashed edges are the four labelled
free cusp channels.}
\label{fig:F3-standard-core}
\end{figure}
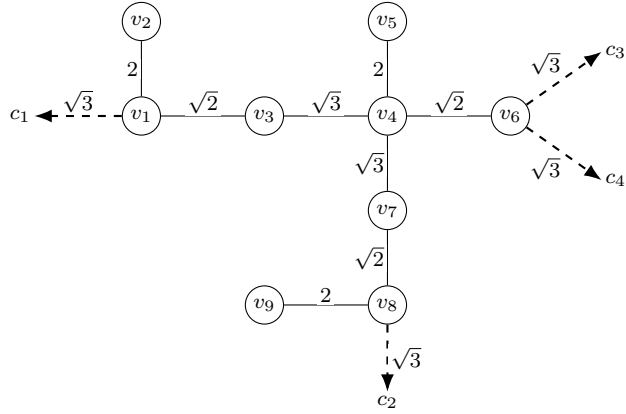

With the first free vertices indexed by \(n=1\), equations
\eqref{eq:cusp-width-definition} and
\eqref{eq:F3-attachment-stabilizer-orders}, with \(q=3\), give
\begin{equation}\label{eq:F3-widths}
 \begin{aligned}
 (w_1,w_2,w_3,w_4)
 &=\frac34\left(\frac1{18},\frac1{18},\frac1{6},\frac1{6}\right),\\
 W&=\frac1{24}\operatorname{diag}(1,1,3,3).
 \end{aligned}
\end{equation}
This records the chosen cusp origins; a different cut transforms the answer by Remark~\ref{rem:height-shift}.

\subsubsection{Finite arithmetic input}

The following proposition records the finite matrix in the present
normalization.  Its resonance polynomial and the classification of its zeros
are results of Arends, Peterson, and Weich
\cite[Sec.~7.7]{ArendsPetersonWeich}; they are retained here so that the
coefficient-level and normalization-exact Plancherel formulas derived below
can be stated without a normalization gap.

\begin{proposition}[Finite resonance data recalled from previous work]
\label{prop:F3-resonance-input}
For the elliptic quotient \eqref{eq:F3-curve},
\begin{equation}\label{eq:F3-F}
 F_X(\mu)=H_C-\sqrt3(\mu+\mu^{-1})I_9
 +\sqrt3\mu^{-1}C_{\mathrm{att}}C_{\mathrm{att}}^*,
\end{equation}
and
\begin{equation}\label{eq:F3-detF}
 \det F_X(\mu)
 =-27\sqrt3\,\mu^{-5}
 (\mu^2+1)^2(\mu-1)^2(\mu+1)^2
 (\mu^2-3)(3\mu^4+1).
\end{equation}
\end{proposition}

\begin{proof}
Formula \eqref{eq:F3-F} is \eqref{eq:F-zeta} with \(q=3\) and
\eqref{eq:F3-attachment}.  The non-monomial factorization is the resonance
polynomial computed in \cite[Sec.~7.7]{ArendsPetersonWeich}.  The scalar and
monomial in \eqref{eq:F3-detF}, which depend on the present normalization,
follow by direct sparse elimination in the order
\(v_2,v_5,v_9,v_3,v_7\).
\end{proof}

\subsubsection{Eisenstein scattering coefficients}

Set
\begin{align*}
 \Delta(\mu)&=(\mu^2-3)(3\mu^4+1),\\
 p(\mu)&=3\mu^6-7\mu^4+3\mu^2-3,\\
 \beta_0(\mu)&=2(\mu^2+1),\qquad
 \beta_1(\mu)=2\mu(\mu^2+1),\qquad
 \beta_2(\mu)=2\mu^2(\mu^2+1).
\end{align*}

\begin{theorem}[Four-cusp Eisenstein scattering]\label{thm:F3-scattering}
For the elliptic quotient \eqref{eq:F3-curve}, the Jacobi-normalized
scattering matrix is
\begin{equation}\label{eq:F3-S}
 S_X(\mu)=\frac1{\Delta(\mu)}
 \begin{pmatrix}
 p(\mu)/\mu^2&\beta_0(\mu)&\beta_1(\mu)&\beta_1(\mu)\\
 \beta_0(\mu)&p(\mu)/\mu^2&\beta_1(\mu)&\beta_1(\mu)\\
 \beta_1(\mu)&\beta_1(\mu)&\beta_2(\mu)&p(\mu)\\
 \beta_1(\mu)&\beta_1(\mu)&p(\mu)&\beta_2(\mu)
 \end{pmatrix},
\end{equation}
and hence
\begin{equation}\label{eq:F3-detS}
 \det S_X(\mu)
 =-\frac{(3\mu^2-1)(\mu^4+3)}
 {\mu^4(\mu^2-3)(3\mu^4+1)}.
\end{equation}
The classical constant-term matrix and the Efrat-holomorphic coefficient
matrix are
\begin{equation}\label{eq:F3-automorphic-matrices}
 \mathcal S_{\mathrm{aut},X}(s)=D_0^{-1}S_X(\mu)D_0,
 \qquad
 \Phi_X(s)=\frac{c_{\mathrm{in}}(\mu)}{c_{\mathrm{out}}(\mu)}
 D_0^{-1}S_X(\mu)D_0,
 \qquad
 D_0=\operatorname{diag}(1,1,\sqrt3,\sqrt3).
\end{equation}
\end{theorem}

\begin{proof}
Substitution of \eqref{eq:F3-Hcore}--\eqref{eq:F3-attachment} into
\eqref{eq:S-from-F}, followed by cancellation of common factors, gives
\eqref{eq:F3-S}.  Its determinant is \eqref{eq:F3-detS}; equivalently it
follows from Proposition~\ref{prop:F3-resonance-input} and Theorem
\ref{thm:determinant-identity}.  Finally, \eqref{eq:F3-widths} and Theorem
\ref{thm:normalization-dictionary} give
\eqref{eq:F3-automorphic-matrices}; the scalar factor \(24^{-1/2}\) cancels in
the conjugation.
\end{proof}

The coefficient matrix contains more information than \eqref{eq:F3-detS}.  For example,
\begin{equation}\label{eq:F3-decoupled-channels}
 S_X(\mu)\frac{(1,-1,0,0)^t}{\sqrt2}
 =\mu^{-2}\frac{(1,-1,0,0)^t}{\sqrt2},
 \qquad
 S_X(\mu)\frac{(0,0,1,-1)^t}{\sqrt2}
 =-\frac{(0,0,1,-1)^t}{\sqrt2}.
\end{equation}
The second identity is forced by the two eastern cusps sharing the same
attachment vertex: their antisymmetric combination lies in
\(\ker C_{\mathrm{att}}\) and has the constant scattering coefficient
\(-1\).  All nontrivial arithmetic poles occur in the remaining
two-dimensional symmetric channel block.

\begin{corollary}[Two exact Maass--Selberg channels in the elliptic example]
\label{cor:F3-maass-selberg-channels}
Put
\[
 v_W=\frac{(1,-1,0,0)^t}{\sqrt2},\qquad
 v_E=\frac{(0,0,1,-1)^t}{\sqrt2},
\]
and, for \(\mu=e^{i\theta}\), let
\[
 \mathcal M_X(\theta)
 =iS_X(e^{i\theta})^*\partial_\theta S_X(e^{i\theta}).
\]
Then
\begin{equation}\label{eq:F3-maass-selberg-eigenchannels}
 \mathcal M_X(\theta)v_W=2v_W,
 \qquad
 \mathcal M_X(\theta)v_E=0.
\end{equation}
The orthogonal complement of \(\operatorname{span}\{v_W,v_E\}\) is the
two-dimensional symmetric channel space on which the remaining arithmetic
mixing occurs.  In particular, the two scalar identities in
\eqref{eq:F3-maass-selberg-eigenchannels} distinguish channel behavior that
is collapsed in the single quantity
\(i\partial_\theta\log\det S_X(e^{i\theta})\).
\end{corollary}

\begin{proof}
The eigenvectors in \eqref{eq:F3-decoupled-channels} are independent of
\(\theta\), with scattering eigenvalues \(e^{-2i\theta}\) and \(-1\),
respectively.  Hence
\[
 i\overline{e^{-2i\theta}}\,\partial_\theta e^{-2i\theta}=2,
 \qquad
 i\overline{(-1)}\,\partial_\theta(-1)=0.
\]
This proves \eqref{eq:F3-maass-selberg-eigenchannels}; the last assertion is
the trace identity \eqref{eq:trace-mass-determinant}.
\end{proof}

\subsubsection{Hasse--Weil, core-supported, and threshold factors}

Since \(\#X(\mathbb F_3)=4\), the Frobenius trace is zero and the Hasse--Weil
zeta function of \(X\) is
\begin{equation}\label{eq:F3-zeta}
        Z_X(u)=\frac{1+3u^2}{(1-u)(1-3u)}.
\end{equation}
As observed in \cite[Sec.~7.7]{ArendsPetersonWeich}, the Hasse--Weil
numerator evaluated at \(u=\mu^2\) is exactly
\(1+3\mu^4\), the last factor of \eqref{eq:F3-detF}.  Its four roots have modulus \(3^{-1/4}\) and are genuine poles of \eqref{eq:F3-S} and \eqref{eq:F3-detS}.  The factorization has the following spectral meaning.
Here a threshold zero means a zero at the branch point \(\mu=\pm1\) that does
not give an \(L^2\)-eigenfunction in this example.

\begin{center}
\small
\begin{tabular}{@{}c|c|c@{}}
factor of \(\det F_X\) & spectral meaning & contribution to \(\det S_X\)\\
\hline
\(\mu^2-3\) & \(L^2\)-eigenvalues \(\pm4\) of \(A\) & poles; reciprocal zeros\\
\((\mu^2+1)^2\) & two core eigenfunctions with eigenvalue \(0\) & cancels identically\\
\((\mu^2-1)^2\) & threshold zeros \(\mu=\pm1\) & cancels in the divisor\\
\(3\mu^4+1\) & resonances from the Hasse--Weil numerator
              & poles; reciprocal zeros
\end{tabular}
\end{center}

The core-supported eigenspace can be checked without a determinant.  In the
vertex order of \eqref{eq:F3-Hcore}, it is spanned by
\begin{equation}\label{eq:F3-core-supported-vectors}
\begin{aligned}
 u_1&=(0,\sqrt6/3,-2\sqrt3/3,0,1,0,0,0,0)^t,\\
 u_2&=(0,-1,\sqrt2,0,0,0,-\sqrt2,0,1)^t.
\end{aligned}
\end{equation}
Indeed, \(H_Cu_j=0\) and \(C_{\mathrm{att}}^*u_j=0\).  Hence these vectors extend by zero to compactly supported \(L^2\)-eigenfunctions and account for \((\mu^2+1)^2\).

Although the two threshold parameters are zeros of \(\det F_X\), the common factor cancels from every entry of \eqref{eq:F3-S}.  The meromorphic continuation has the finite limits
\begin{equation}\label{eq:F3-threshold-limits}
 S_X(\sigma)=\frac12
 \begin{pmatrix}
 1&-1&-\sigma&-\sigma\\
 -1&1&-\sigma&-\sigma\\
 -\sigma&-\sigma&-1&1\\
 -\sigma&-\sigma&1&-1
 \end{pmatrix},
 \qquad \sigma\in\{1,-1\}.
\end{equation}
Thus a threshold resonance may be invisible to the scattering divisor while leaving a nontrivial limiting matrix.  This is precisely the distinction between a resonance list, a determinant, and the coefficient-level scattering operator.

\subsubsection{Spectral expansion}

For the chosen cusp origins, \eqref{eq:F3-widths} makes the Efrat-holomorphic Plancherel density completely explicit:
\begin{equation}\label{eq:F3-Plancherel-density}
 P_{E,X}(\theta)\,d\theta
 =\frac{9}{4\pi(1+3\sin^2\theta)}
 \operatorname{diag}\left(1,1,\frac13,\frac13\right)d\theta.
\end{equation}
The constant-term-normalized family has the constant density
\begin{equation}\label{eq:F3-constant-term-density}
 \frac3\pi\operatorname{diag}\left(1,1,\frac13,\frac13\right)d\theta.
\end{equation}
Combining the resonance classification in \cite[Sec.~7.7]{ArendsPetersonWeich}
with Corollary~\ref{thm:arithmetic-decomposition} yields
\begin{equation}\label{eq:F3-spectrum}
 L^2(\Gamma\backslash\mathcal T,\lambda)
 =\ker(A-4)\oplus\ker(A+4)\oplus\ker A\oplus L^2_{\mathrm{Eis}},
\end{equation}
where the first two summands have dimension one, \(\dim\ker A=2\), and
\(L^2_{\mathrm{Eis}}\) has spectrum \([-2\sqrt3,2\sqrt3]\) of multiplicity
four.  For a spectral value \(\xi\), let \(P_\xi^A\) denote the orthogonal
projection onto \(\ker(A-\xi)\).  For
\(f\in L^2(\Gamma\backslash\mathcal T,\lambda)\), the automorphic inversion
formula is
\begin{equation}\label{eq:F3-full-expansion}
 f=P_{4}^Af+P_{-4}^Af+P_{0}^Af
 +\int_0^\pi E_s^{\mathrm{hol}}P_{E,X}(\theta)
 (E_s^{\mathrm{hol}})^*f\,d\theta,
 \qquad s=\frac12+\frac{i\theta}{\log3}.
\end{equation}
For general \(f\), the Eisenstein integral is understood through the unitary
transform of Theorem~\ref{thm:unitary-eisenstein-transform}, rather than as a
pointwise pairing with a generalized eigenfunction; conjugation by \(U\)
gives the standard Jacobi expansion \eqref{eq:complete-standard-resolution}.
Here \(\operatorname{rank}P_4^A=\operatorname{rank}P_{-4}^A=1\) and
\(\operatorname{rank}P_0^A=2\).  The resonance set and its classification are
recalled from \cite[Sec.~7.7]{ArendsPetersonWeich}.  The coefficient matrix
\eqref{eq:F3-S}, its two automorphic normalizations, the symmetry decomposition
\eqref{eq:F3-decoupled-channels}, the channel values
\eqref{eq:F3-maass-selberg-eigenchannels}, and the normalized Plancherel
expansion are the additional conclusions obtained here.

\section{Conclusion and scope}\label{sec:consequences}

Theorem~\ref{thm:unitary-eisenstein-transform} identifies the absolutely
continuous spherical automorphic space unitarily with the \(r\)-channel
spectral space in \eqref{eq:intro-unitary-transform}, with exact stabilizer
and cusp-width normalization;
Corollary~\ref{cor:adelic-unitary-identification}
assembles the component transforms at fixed finite level.  The accompanying
matrix Maass--Selberg identity recovers
\(iS^*\partial_\theta S\) from truncated automorphic Gram matrices.  Its
traceless part, made explicit for \(\Gamma_0(T)\) and the four-cusp elliptic
quotient, is channel-resolved information unavailable from the scattering
determinant alone.

The finite resonance matrix and elliptic factorization of Arends, Peterson,
and Weich are used as inputs \cite{ArendsPetersonWeich}; adjoining the labelled
attachment map produces the Eisenstein coefficient matrix and separates
core-supported eigenfactors from scattering poles.  The present paper focuses on the rank-one cusp setting. 
Natural directions for further study include quotients with noncuspidal regular ends, possible extensions to higher-rank buildings, and a more systematic analysis of threshold phenomena at \(\zeta=\pm1\).






\begingroup
\linespread{1.15}\selectfont

\endgroup

\enlargethispage{2\baselineskip}

\begin{thebibliography}{99}
\setlength{\parsep}{0pt}

\bibitem{ArendsPetersonWeich}
C. Arends, C. Peterson, and T. Weich,
\emph{Resonances on geometrically finite graphs},
\href{https://arxiv.org/abs/2603.26443}{arXiv:2603.26443v2} [math.SP], 2026.

\bibitem{BassLubotzky}
H. Bass and A. Lubotzky,
\emph{Tree Lattices},
Progress in Mathematics, vol. 176, Birkh\"auser, 2001.

\bibitem{Bravo}
C. Bravo,
Quotients of the Bruhat--Tits tree by function field analogs of the Hecke congruence subgroups,
\emph{J. Number Theory} \textbf{259} (2024), 171--218,
\href{https://doi.org/10.1016/j.jnt.2023.12.010}{doi:10.1016/j.jnt.2023.12.010}.

\bibitem{Cartier}
P. Cartier,
Harmonic analysis on trees,
in \emph{Harmonic Analysis on Homogeneous Spaces} (Williamstown, MA, 1972),
Proc. Sympos. Pure Math., vol.~26, American Mathematical Society,
Providence, RI, 1973, pp.~419--424,
\href{https://doi.org/10.1090/pspum/026/0338272}{doi:10.1090/pspum/026/0338272}.

\bibitem{Chekhov}
L. Chekhov,
\(L\)-functions in scattering on \(p\)-adic multiloop surfaces,
\emph{J. Math. Phys.} \textbf{36} (1995), no. 1, 414--425,
\href{https://doi.org/10.1063/1.531315}{doi:10.1063/1.531315}.

\bibitem{ChekhovSurvey}
L. O. Chekhov,
A spectral problem on graphs and \(L\)-functions,
\emph{Russian Math. Surveys} \textbf{54} (1999), no.~6, 1197--1232,
\href{https://doi.org/10.1070/RM1999v054n06ABEH000231}
{doi:10.1070/RM1999v054n06ABEH000231}.

\bibitem{ColinTruc}
Y. Colin de Verdi\`ere and F. Truc,
Scattering theory for graphs isomorphic to a regular tree at infinity,
\emph{J. Math. Phys.} \textbf{54} (2013), 063502,
\href{https://doi.org/10.1063/1.4807310}{doi:10.1063/1.4807310}.

\bibitem{DeitmarKang}
A. Deitmar and M.-H. Kang,
Tree-lattice zeta functions and class numbers,
\emph{Michigan Math. J.} \textbf{67} (2018), no.~3, 617--645,
\href{https://doi.org/10.1307/mmj/1529460323}
{doi:10.1307/mmj/1529460323}.

\bibitem{EfratSpectralDeformations}
I. Efrat,
Spectral deformations of automorphic functions over graphs of groups,
\emph{Invent. Math.} \textbf{102} (1990), no.~2, 447--462,
\href{https://doi.org/10.1007/BF01233435}{doi:10.1007/BF01233435}.

\bibitem{EfratAutomorphicSpectra}
I. Efrat,
Automorphic spectra on the tree of \(\mathrm{PGL}_2\),
\emph{Enseign. Math. (2)} \textbf{37} (1991), no. 1--2, 31--43,
\href{https://doi.org/10.5169/seals-58728}{doi:10.5169/seals-58728}.

\bibitem{FigaTalamancaNebbia}
A. Fig\`a-Talamanca and C. Nebbia,
\emph{Harmonic Analysis and Representation Theory for Groups Acting on
Homogeneous Trees},
London Mathematical Society Lecture Note Series, vol.~162,
Cambridge University Press, Cambridge, 1991,
\href{https://doi.org/10.1017/CBO9780511662324}{doi:10.1017/CBO9780511662324}.

\bibitem{Golinskii}
L. Golinskii,
Spectra of infinite graphs with tails,
\emph{Linear Multilinear Algebra} \textbf{64} (2016), 2270--2296,
\href{https://doi.org/10.1080/03081087.2016.1155529}{doi:10.1080/03081087.2016.1155529}.

\bibitem{Harder}
G. Harder,
Minkowskische Reduktionstheorie {\"u}ber Funktionenk{\"o}rpern,
\emph{Invent. Math.} \textbf{7} (1969), 33--54,
\href{https://doi.org/10.1007/BF01418773}{doi:10.1007/BF01418773}.

\bibitem{HarderAutomorphic}
G. Harder,
Chevalley groups over function fields and automorphic forms,
\emph{Ann. of Math. (2)} \textbf{100} (1974), no.~2, 249--306,
\href{https://doi.org/10.2307/1971073}{doi:10.2307/1971073}.

\bibitem{HongKwonZeta}
S. Hong and S. Kwon,
Zeta functions of geometrically finite graphs of groups,
\emph{Ann. Comb.} \textbf{29} (2025), 995--1018,
\href{https://doi.org/10.1007/s00026-025-00759-w}
{doi:10.1007/s00026-025-00759-w}.

\bibitem{JacquetLanglands}
H. Jacquet and R. P. Langlands,
\emph{Automorphic Forms on \(\mathrm{GL}(2)\)},
Lecture Notes in Mathematics, vol. 114, Springer, 1970.

\bibitem{Kato}
T. Kato,
\emph{Perturbation Theory for Linear Operators},
Classics in Mathematics, Springer, 1995.

\bibitem{Knudson}
K. P. Knudson,
Integral homology of \(\mathrm{PGL}_2\) over elliptic curves,
in \emph{Algebraic \(K\)-Theory} (Seattle, WA, 1997),
Proc. Sympos. Pure Math., vol.~67, American Mathematical Society,
Providence, RI, 1999, pp.~175--180.

\bibitem{LanglandsEisenstein}
R. P. Langlands,
\emph{On the Functional Equations Satisfied by Eisenstein Series},
Lecture Notes in Mathematics, vol. 544, Springer, 1976.

\bibitem{LaxPhillips}
P. D. Lax and R. S. Phillips,
\emph{Scattering Theory for Automorphic Functions},
Annals of Mathematics Studies, vol.~87, Princeton University Press,
Princeton, NJ, 1976.

\bibitem{LiEisenstein}
W.-C. W. Li,
Eisenstein series and decomposition theory over function fields,
\emph{Math. Ann.} \textbf{240} (1979), 115--139,
\href{https://doi.org/10.1007/BF01364628}{doi:10.1007/BF01364628}.

\bibitem{Nagao}
H. Nagao,
On \(\mathrm{GL}(2,K[x])\),
\emph{J. Inst. Polytech. Osaka City Univ. Ser. A} \textbf{10} (1959), 117--121.

\bibitem{NagoshiSpectra}
H. Nagoshi,
Spectra of arithmetic infinite graphs and their application,
\emph{Interdiscip. Inform. Sci.} \textbf{7} (2001), no.~1, 67--76,
\href{https://doi.org/10.4036/iis.2001.67}{doi:10.4036/iis.2001.67}.

\bibitem{Novikov}
S. P. Novikov,
Discrete Schr\"odinger operators and topology,
\emph{Asian J. Math.} \textbf{2} (1998), no. 4, 921--934.

\bibitem{PaulinGeometricallyFinite}
F. Paulin,
Groupes g\'eom\'etriquement finis d'automorphismes d'arbres et
approximation diophantienne dans les arbres,
\emph{Manuscripta Math.} \textbf{113} (2004), no.~1, 1--23,
\href{https://doi.org/10.1007/s00229-003-0413-1}
{doi:10.1007/s00229-003-0413-1}.

\bibitem{ReedSimonIII}
M. Reed and B. Simon,
\emph{Methods of Modern Mathematical Physics. III: Scattering Theory},
Academic Press, 1979.

\bibitem{RomanovRudinBT}
R. V. Romanov and G. E. Rudin,
Scattering on the Bruhat--Tits tree. I,
\emph{Phys. Lett. A} \textbf{198} (1995), 113--118,
\href{https://doi.org/10.1016/0375-9601(94)00997-4}{doi:10.1016/0375-9601(94)00997-4}.

\bibitem{RomanovRudinPadic}
R. V. Romanov and G. E. Rudin,
Scattering on \(p\)-adic graphs,
\emph{Comput. Math. Appl.} \textbf{34} (1997), no. 5--6, 587--597,
\href{https://doi.org/10.1016/S0898-1221(97)00155-7}{doi:10.1016/S0898-1221(97)00155-7}.

\bibitem{Serre}
J.-P. Serre,
\emph{Trees},
Springer Monographs in Mathematics, Springer, 2003.

\bibitem{SmithLifetime}
F. T. Smith,
Lifetime matrix in collision theory,
\emph{Phys. Rev.} \textbf{118} (1960), no.~1, 349--356,
\href{https://doi.org/10.1103/PhysRev.118.349}%
{doi:10.1103/PhysRev.\allowbreak118.349}.

\bibitem{Takahashi}
S. Takahashi,
The fundamental domain of the tree of \(\mathrm{GL}_2\) over the function field of an elliptic curve,
\emph{Duke Math. J.} \textbf{72} (1993), no.~1, 85--97,
\href{https://doi.org/10.1215/S0012-7094-93-07204-3}{doi:10.1215/S0012-7094-93-07204-3}.

\bibitem{Teschl}
G. Teschl,
\emph{Jacobi Operators and Completely Integrable Nonlinear Lattices},
Mathematical Surveys and Monographs, vol. 72, American Mathematical Society, 2000.

\bibitem{VarbanovBrun}
M. Varbanov and T. A. Brun,
Quantum scattering theory on graphs with tails,
\emph{Phys. Rev. A} \textbf{80} (2009), 052330,
\href{https://doi.org/10.1103/PhysRevA.80.052330}{doi:10.1103/PhysRevA.80.052330}.

\bibitem{Yafaev}
D. R. Yafaev,
\emph{Mathematical Scattering Theory: General Theory},
Translations of Mathematical Monographs, vol. 105, American Mathematical Society, 1992.

\end{thebibliography}
\end{document}